%% file: main.tex
\documentclass[a4paper,11pt,reqno]{amsart}
\usepackage{amssymb,amsmath}
\usepackage[utf8]{inputenc}
\usepackage[T2A]{fontenc}
\usepackage[russian, english]{babel}
\usepackage[utf8]{inputenc}
\usepackage{graphicx}
\usepackage{amsmath,amsfonts,amssymb,amsthm,mathtools, dsfont} 
\usepackage{wasysym}
\usepackage[margin=1in]{geometry}
\usepackage{quiver}
\usetikzlibrary{angles,calc,quotes,babel,arrows.meta,decorations.markings,plotmarks}
\makeatletter
\tikzcdset{
  eq node/.style={
    commutative diagrams/math mode=false, anchor=center},
  eq/.style={
    phantom,
    /tikz/every to/.append style={
      edge node={node[commutative diagrams/eq node]
        {\@eqnswtrue\make@display@tag\ltx@label{#1}}}}}}
\makeatother

\graphicspath{{pictures/}}
\DeclareGraphicsExtensions{.pdf,.png,.jpg}
\usepackage{indentfirst}
\usepackage{textcomp}
\usepackage{wrapfig}
\usepackage[unicode, pdftex]{hyperref}
\title{Secondary Stiefel--Whitney numbers and corresponding cobordism groups}
\author{Viktor Lavrukhin}
\address{PDMI RAS, 27 Fontanka quay, 191023, Saint Petersburg, Russia}
\email{}

\newcommand\Z{\mathbb{Z}}

\newcommand\R{\mathbb{R}}
\renewcommand\Theta{\vartheta}
\newcommand\id{\mathrm{id}}

\newcommand\Sq{\mathrm{Sq}}
\newcommand\hofib{\mathrm{hofib}}

\newcommand\Gr{\mathrm{Gr}}
\DeclareMathOperator{\Th}{Th}
\newcommand{\defeq}{\vcentcolon=}

\newcommand{\cob}{\mathsf{cob}}

\newtheorem{theorem}{Theorem}

\newtheorem{lemma}[theorem]{Lemma}
\newtheorem{corollary}[theorem]{Corollary}
\newtheorem{proposition}[theorem]{Proposition}
\theoremstyle{definition}
\newtheorem{definition}[theorem]{Definition}
\newtheorem{remark}[theorem]{Remark} 
\newtheorem{example}[theorem]{Example}

\numberwithin{theorem}{section} 
\numberwithin{equation}{section}

\makeatletter
\@namedef{subjclassname@2020}{\textup{2020} Mathematics Subject Classification}
\makeatother

\begin{document}
\maketitle
\begin{abstract}
For every relation $R$ between Stiefel--Whitney numbers of closed $(n+1)$-manifolds we consider an associated invariant $\varkappa_R$ of null-cobordant $n$-manifolds with a certain additional structure. For $n=2k-1$ and $R = w_{n+1}+v_k^2$ the invariant \(\varkappa_R\) equals the Kervaire semi-characteristic. In addition, we construct the cobordism group $\Omega_n^R$, which extends the unoriented cobordism group $\Omega_n^O$. We show that \(\varkappa_R\) is a complete invariant of \(R\)-cobordism classes of null-cobordant \(n\)-manifolds. We prove that our invariant $\varkappa_R$ and \(R\)-cobordism class of manifold are quadratic in the sense of Gusarov--Vassiliev--Podkorytov. 
\end{abstract}

\section{Introduction}
The theory of Gusarov--Vassiliev finite-type invariants provides a powerful framework for studying knots and links. At the core, this theory provides a filtration on the set of invariant maps \(\mathrm{Emb}(S^1, S^3) \to V\), where \(V\) denotes an abelian group. The term `invariant' means that the map respects the isotopy relation on the space of knots.
It turns out that this idea can be generalized in many ways by changing the domain of an invariant to some version of the mapping space between the two space-like objects.
S.\,S. Podkorytov has initiated a project of developing related methods in various homotopical contexts, as presented in, e.g., \cite{Pod3, Pod4, Pod5}.

In \cite{Ker1}, Kervaire defined the semi-characteristic \(K(X)\) for every closed \((2k-1)\)-dimensional smooth manifold \(X\) by the formula
\[K(X):=\frac{\operatorname{dim}H^*(X;\Z/2)}{2}\bmod2.\]
We can gather all the manifolds in one space using the concept of a germ. Choose a sufficiently large \(q\) and consider the set \(E\) of germs of the embedded \((2n-1)\)-manifolds in \(\R^q\). Then, every submanifold \(X \subset \R^q\) can be represented by its characteristic function \(I_X \colon E \to \Z/2\). We say that an invariant of a smooth manifold \(X \mapsto \varphi(X) \in V\) is polynomial of degree \(k\), if there is a polynomial function \(\Phi \colon (\Z/2)^E \to V\) such that \(\Phi(I_X) = \varphi(X)\) for every closed smooth submanifold \(X \subset \R^q\). To clarify the idea, let us provide the following example. For \(j=1,\dots,5\) consider five oriented circles \(S^1_j\) embedded in \(\R^3\) as in Figure~\ref{fig:mpr}.
\begin{figure}[h!]
    \centering
    \label{fig:mpr}
    \includegraphics[width=0.5\textwidth]{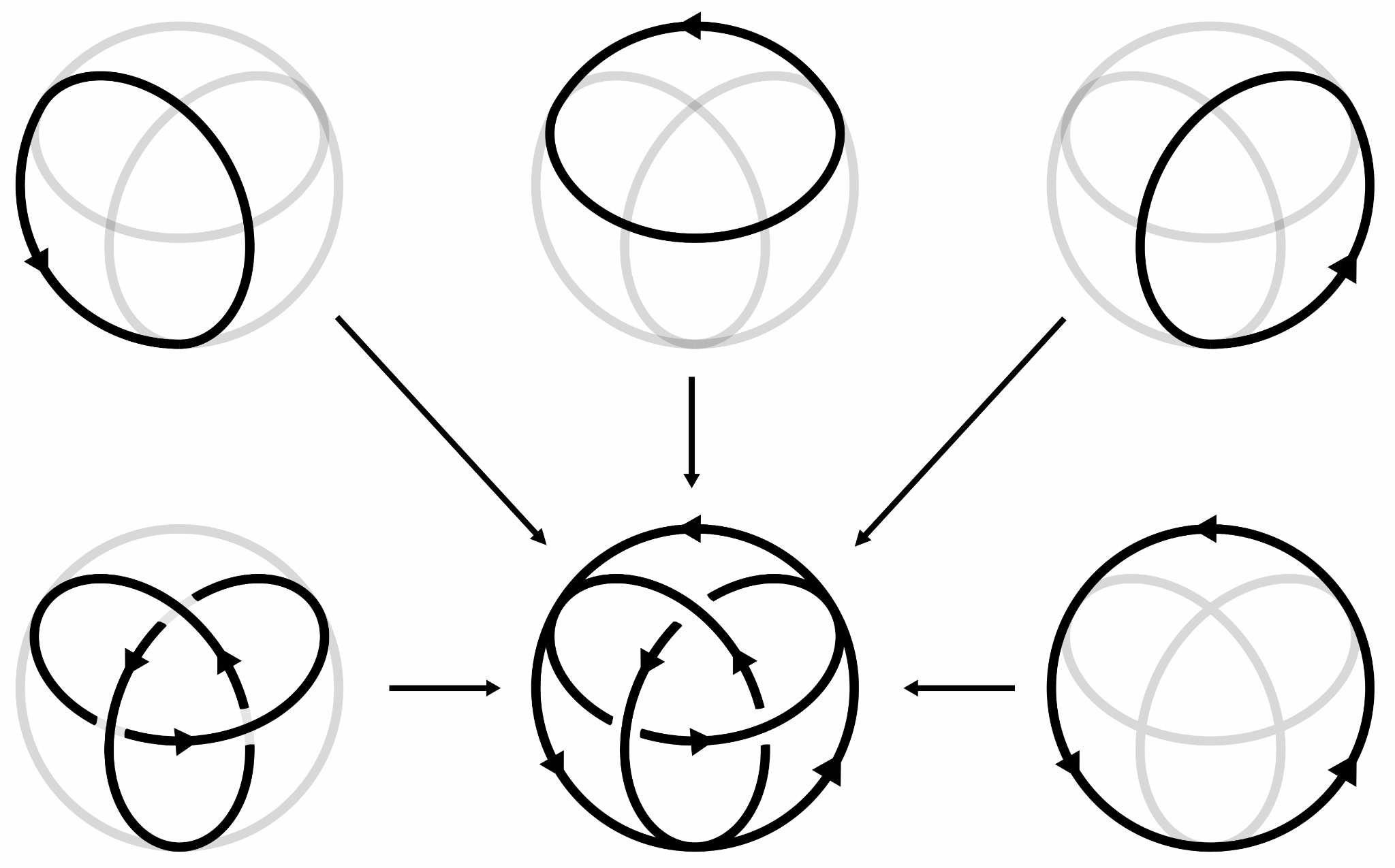}
\caption{Five embedded circles into 3-dimensional space.}
\end{figure}
The characteristic functions of the \(S^1_j\)'s are linearly dependent:
$$\sum_{j=1}^{5} I_{S^1_j} = 0.$$
On the other hand, the Kervaire semi-characteristic does not vanish:
$$\sum_{j=1}^{5} K(S^1_j) = 1.$$
It follows that \(K(-)\) is not linear.
In \cite{Pod1}, S.\,S. Podkorytov establishes that the Kervaire semi-characteristic of a closed \((2n-1)\)-manifold with \(v_n\)-structure (so-called Wu structure) is quadratic. In this paper, our aim is to generalize his result.
It is known that the Stiefel--Whitney numbers of closed \((n+1)\)-manifolds are not independent. For example, for an odd-dimensional closed \((n+1)\)-manifold \(Z\) one has \(w_{n+1}(Z) = 0\).
A less-known relation is \(w_{n+1}(Z) + v_{(n+1)/2}(Z)^2 = 0\) for even-dimensional \((n+1)\)-manifolds \(Z\), where \(v_i(Z)\) denotes the Wu class. Let \(\mathcal{R}_{n+1}\) denote a subgroup of \(H^{n+1}(BO_{n+1})\) consisting of all the universal relations on Stiefel--Whitney numbers of closed \((n+1)\)-manifolds. This group is described in papers \cite{Dold, Br}.
Of course, the class corresponding to \(R\) vanishes on every closed \(n\)-dimensional manifold \(X\) as well for dimension reasons.
With an arbitrary relation \(R \in \mathcal{R}_{n+1}\) we can associate a fibration \(f \colon B \to \Gr_n(\R^q)\) such that a \((B, f)\)-structure on a given \(n\)-dimensional closed manifold \(X\) would correspond to a choice of an `annihilation' for the class \(R\) in \(H^*(X)\); see \S\ref{sec:generalized_semi-characteristic} for a precise definition. We call such \((B, f)\)-manifolds \textit{\(R\)-manifolds}.

Let \(R \in \mathcal{R}_{n+1}\) be some relation on the Stiefel--Whitney numbers. For a bounding manifold \(X\), equipped with some \(R\)-structure \(g_X\), we define its invariant \(\varkappa_R(X, g_X)\) taking values in \(\Z/2\), which we call the \textit{secondary Stiefel--Whitney number}; see Definition~\ref{def:generalized_semicharacteristic}. We prove that it is well-defined; see Lemma~\ref{ind} and Proposition~\ref{prop:lift_independence}. This invariant generalizes the Kervaire semi-characteristic \(K\), as can be seen from the following theorem (note that every \(n\)-manifold admits some \(v_{(n+1)/2}\)-structure).
\begin{proposition}[Proposition~\ref{th:our_is_kervaire}]
\label{thm:1}
Let \(n\) be odd and \(R = w_{n+1} + v^2_{(n+1)/2}\). Then the $v_{(n+1)/2}$-structure on bounding \(n\)-manifold \(X\) gives rise to an \(R\)-structure \(g_X\) such that the equality \[\varkappa_R(X, g_X) = K(X)\] holds.
\end{proposition}

Then, we associate a cobordism group \(\Omega^R_n\) with every \(R\); see Definition~\ref{def:R_cobordism_group}. It extends the unoriented cobordsim group \(\Omega^O_n\), and \(\varkappa_R\) serves as a complete invariant that determines an \(R\)-cobordism class.
\begin{theorem}[Corollary~\ref{cor:cobordism_group_descr}]
\label{thm:2}
The forgetful homomorphism \(U \colon \Omega^R_n \to \Omega^O_n\) is surjective, and the secondary Stiefel--Whitney number yields an isomorphism \(\varkappa_R \colon \ker(U) \to \Z/2\). Thus, there is a short exact sequence
\[0 \to \Z/2 \to \Omega^R_n \to \Omega^O_n \to 0.\]
\end{theorem}

Lastly, we show that the \(R\)-cobordism class depends on an \(R\)-manifold quadratically in the above sense.
The claim that the correspondence \((X, g_X) \mapsto [X, g_X]^R\) is quadratic can be formalized in the following way.
Let \(q \ge 2n+2\). Denote the set of germs of embedded \(R\)-manifolds in \(\R^q\) by \(E_R\); see \S\ref{sec:quadr} for the precise definition. Then the abelian group \((\Z/2)^{E_R}\) of \(\Z/2\)-valued functions on \(E_R\) naturally contains the set of embedded \(R\)-manifolds via characteristic functions:
\[\{R\text{-submanifolds of}\;\R^q\} \hookrightarrow (\Z/2)^{E_R}, \quad (X, g_X) \mapsto I_{(X, g_X)}.\]
\begin{theorem}[Theorem~\ref{main}]
\label{thm:3}
There is a quadratic function \(\mathfrak{Q} \colon (\Z/2)^{E_R} \to \Omega^R_n\) such that for any closed \(R\)-manifold \((X, g_X)\) the equality \(\mathfrak{Q}(I_{(X, g_X)}) = [X, g_X]^R\) holds.
\end{theorem}
This implies that \(\varkappa_R \colon (X, g_X) \mapsto \varkappa_R(X, g_X) \in \Z/2\) is quadratic in the same sense.

To prove Theorem~\ref{thm:3}, we develop some geometric machinery.  First, we use a version of Pontryagin--Thom collapse map for \(R\)-manifolds; see \S\ref{sec:pontryagin-thom_collapse}. Second, we prove a relative version of Thom isomorphism and show how it interacts with characteristic classes; see \S\ref{sec:relative_thom}. In the end, all this leads to a geometric expression of \(\varkappa_R\) via Steenrod operations applied to certain carefully chosen cochains; see \S\ref{sec:geometric_formula_for_semichar}. We use the fact that Steenrod operations depend on cochains quadratically, to show that \(\varkappa_R\) depends on \(R\)-manifolds quadratically as well.

\subsection*{Organization}
In Section~\ref{sec:preliminaries}, we declare our notational conventions and recall some basic properties for (co)homology, vector bundles, characteristic classes, and \((B, f)\)-manifolds. In Section~\ref{sec:generalized_semi-characteristic}, we define \(R\)-manifolds and construct an invariant of bounding \(R\)-manifolds, which we call the secondary Stiefel--Whitney number \(\varkappa_R(-)\). We prove its basic properties and show that it equals the Kervaire semi-characteristic. In Section~\ref{sec:cobordism_group}, we define the cobordism group \(\Omega^R_n\) and show that it extends \(\Omega^O_n\) with the kernel isomorphic to \(\Z/2\). We also provide some examples related to \(\mathrm{Pin}^-\)-manifolds in low dimensions. In Section~\ref{sec:geometric_constructions}, we develop all the necessary geometry to establish the quadratic property.
In Section~\ref{sec:quadr}, we recall the setting of smooth germs of structured manifolds and prove the quadratic property.

\subsection*{Acknowledgements}
I express sincere gratitude to Dr. S.\,S. Podkorytov, for proposing the problem and for invaluable support and guidance throughout all the research. Special thanks go to my friends V. Ionin for his diligent review of the text and A. Miller for providing the illustrations. 

Research is supported by PJSC «Gazprom Neft», aggr. \(\#\) ГПН-26/03000/00502/Р.
\maketitle
\input{prel}
\input{def.tex}

\input{cob}
\input{geometric_model}
\input{quadr}

\end{document}

%% file: prel.tex
\section{Preliminaries}
\label{sec:preliminaries}
\subsection{(Co)homologies}
By default, we denote by \(H_*(-)\) and \(H^*(-)\) the functors of singular homology and cohomology with coefficients \(\Z/2\), respectively. For a continuous map \(f \colon X \to Y\), we denote by \(H_*(X \xrightarrow{f} Y)\) and by \(H^*(X \xrightarrow{f} Y)\) the reduced homology and cohomology of the homotopy cofiber \(\mathrm{hocofib}(f)\) with coefficients \(\Z/2\), respectively.
For a commutative square
\[\begin{tikzcd}[cramped]
X & Y \\
{X'} & {Y'}
\arrow["f", from=1-1, to=1-2]
\arrow["a"', from=1-1, to=2-1]
\arrow["b", from=1-2, to=2-2]
\arrow["{f'}"', from=2-1, to=2-2]
\end{tikzcd}\]
there are induced natural homomorphisms\footnote{If \(A \hookrightarrow X\) is a cofibration, then \(H_*(A \hookrightarrow X) \cong H_*(X, A)\) and \(H^*(A \hookrightarrow X) \simeq H^*(X, A)\).}
\[(a, b)_* \colon H_*(X \xrightarrow{f} Y) \to H_*(X' \xrightarrow{f'} Y')\]
and
\[(a, b)^* \colon H^*(X' \xrightarrow{f'} Y') \to H^*(X \xrightarrow{f} Y).\]
Also, there is a corresponding long exact sequence
\[\dots \to H^{*-1}(X) \xrightarrow{\delta} H^*(X \xrightarrow{f} Y) \xrightarrow{\vert_Y} H^*(Y) \xrightarrow{f^*} H^*(X) \to \dots\]
associated with a map \(f\).

If \(X\) is a closed \(n\)-manifold, then \([X] \in H_{n}(X)\) denotes its fundamental class. For a compact \((n+1)\)-manifold \(Y\) with boundary, there is a relative fundamental class \([Y] \in H_{n+1}(\partial Y \to Y)\).

There is a canonical pairing
\[\langle -, - \rangle \colon H^k(X) \otimes H_k(X) \to \Z/2\]
and a relative version of this pairing
\[\langle -, - \rangle \colon H^k(X \xrightarrow{f} Y) \otimes H_k(X \xrightarrow{f} Y) \to \Z/2.\]

\subsection{Vector bundles and characteristic classes}
Let \(\xi\) be a vector bundle over a topological space \(X\).
Its Stiefel--Whitney classes are denoted by \(w_i(\xi) \in H^i(X)\). Its Wu classes are denoted by \(v_i(\xi) \in H^i(X)\). The Wu classes are defined by the identity \(\Sq(v(\xi)) = w(\xi)\), where
\[\Sq = \sum_{i = 0}^\infty \Sq^i, \quad v(\xi) = \sum_{i=0}^\infty v_i(\xi), \quad w(\xi) = \sum_{i = 0}^\infty w_i(\xi).\]

\subsection{\((B, f)\)-manifolds}
Throughout this article, the term `(sub)manifold' refers to a smooth (sub)manifold. For the non-negative integer \(n\) fix an integer \(q\) such that \(q \ge 2n+2\). Denote by \(\Gr_n(\R^q)\) the \textit{Grassmannian manifold} of \(n\)-dimensional subspaces in \(\R^q\) and by \(\gamma_n^q\) the \textit{tautological bundle} over \(\Gr_n(\R^q)\). Let \(X\) be a compact \(n\)-dimensional submanifold of \(\R^q\). Then its tangent bundle is a subbundle of \(T\R^q\), and is classified by the \textit{Gauss map} \(\tau_X \colon X \to \Gr_n(\R^q)\). 
 
 We say that a compact submanifold \(Y\) of
\[\R^{q+1}_- \defeq \{(x_1, \ldots, x_{q+1}) \in \R^{q+1} \mid x_{q+1} \le 0\}\]
is a \textit{null-cobordism} for a closed\footnote{By a closed submanifold, we mean a submanifold that is compact and without boundary.} submanifold \(X\subseteq \R^q\) (or just \(X\) bounds \(Y\)), if \(\partial Y = X\) and \(Y\) meets the boundary \(\partial\R^{q+1}_- =\R^q\) orthogonally.
There is a corresponding Gaussian map \(\tau_Y \colon Y \to \Gr_{n+1}(\R^{q+1})\) and the following diagram commutes:
\[\begin{tikzcd}[cramped]
	{X=\partial Y} & Y \\
	{\Gr_n(\R^q)} & {\Gr_{n+1}(\R^{q+1}).}
	\arrow[hook, from=1-1, to=1-2]
	\arrow["{\tau_X}"', from=1-1, to=2-1]
	\arrow["{\tau_Y}", from=1-2, to=2-2]
	\arrow[hook, "{i}", from=2-1, to=2-2]
\end{tikzcd}\]

We say that two closed \(n\)-dimensional submanifolds \(X\) and \(X'\) of \(\R^q\) are \textit{cobordant}, if there is a compact \((n+1)\)-dimensional submanifold \(Y\subseteq\R^q\times[0,1]\) such that \(\partial Y = X\times 0 \sqcup X'\times 1\) and \(Y\) meets the boundary \(\partial(\R^{q}\times [0,1]) =(\R^q\times 0) \sqcup (\R^q\times 1)\) orthogonally; such a submanifold \(Y\) is called a \textit{cobordism between \(X\) and \(X'\)}.

Let \(B\) be a topological space, and \(f \colon B \to \Gr_n(\R^q)\) be a fibration.
If \(X\) is a compact submanifold of \(\R^q\), then a \textit{\((B, f)\)-structure} on it is a continuous map \(g_X\colon X \to B\) such that \(f \circ g_X = \tau_X\). An \(n\)-dimensional  \textit{\((B, f)\)-manifold} is a submanifold of \(\R^q\) equipped with some \((B, f)\)-structure.

If there is a commutative square
\[\begin{tikzcd}[cramped]
	B & {B_1} \\
	{\Gr_n(\R^q)} & {\Gr_{n+1}(\R^{q+1}),}
	\arrow["j", from=1-1, to=1-2]
	\arrow["f"', from=1-1, to=2-1]
	\arrow["{f_1}", from=1-2, to=2-2]
	\arrow[hook, "{i}", from=2-1, to=2-2]
\end{tikzcd}\]
then we say that two closed \((B, f)\)-manifolds \((X, g_X)\) and \((X', g_{X'})\) are \textit{\((B_1, f_1)\)-cobordant}, if there is a cobordism \(Y\) between \(X\) and \(X'\)  with \((B_1, f_1)\)-structure \(g_Y\) such that $$g_Y\vert_{X \sqcup X'} = j \circ (g_X \sqcup g_{X'}).$$ For a comprehensive overview, see \cite{Sto1}.

%% file: def.tex
\section{Secondary Stiefel--Whitney number of a bounding $R$-manifold}
\label{sec:generalized_semi-characteristic}
Define \textit{the subgroup of relations between Stiefel--Whitney numbers of closed manifolds} $\mathcal{R}_{n+1}$ to be the subgroup of \(H^{n+1}(\Gr_{n+1}(\R^{q+1}))\) given by
$$\mathcal{R}_{n+1} = \bigcap_{Z}\mathrm{ker}{(\tau_Z^* \colon H^{n+1}(\Gr_{n+1}(\R^{q+1})) \to H^{n+1}(Z))},$$
where the intersection is taken over all closed $(n+1)$-dimensional submanifolds \(Z\subset\R^{q+1}\).
  
For an arbitrary relation $R\in \mathcal{R}_{n+1}$ consider its inverse image $i^*(R)\in H^{n+1}(\Gr_n(\R^q))$ and its classifying map $c\colon \Gr_n(\R^q)\to K(\Z/2,n+1)$. Let $B$ be the pullback of the path space fibration over $K(\Z/2,n+1)$ along the map $c$. Then there is a homotopy fibration sequence
\[K(\Z/2, n) \to B \xrightarrow{f} \Gr_n(\R^q) \xrightarrow{c} K(\Z/2, n+1).\]
The main objects of our interest will be closed $n$-dimensional $(B,f)$-manifolds with $B$ and $f$ as above. For brevity, we will call such \((B, f)\)-manifolds \textit{\(R\)-manifolds}.

\begin{remark}\label{exi}
Every closed $n$-dimensional submanifold $X\subseteq\R^q$ possesses a structure of an \(R\)-manifold. In fact, the only obstruction to the existence of a map $g_X\colon X\to B$ that makes the following diagram commute 
\[\begin{tikzcd}
	& B & {K(\Z/2,n)} \\
	X & {\Gr_n(\R^q)}
	\arrow["f", from=1-2, to=2-2]
	\arrow[from=1-3, to=1-2]
	\arrow["{g_X}", dashed, from=2-1, to=1-2]
	\arrow["{\tau_X}"', from=2-1, to=2-2]
\end{tikzcd}\]
belongs to $H^{n+1}(X;\pi_n(K(\Z/2,n))) = H^{n+1}(X)$, which is trivial since $X$ has dimension $n$.
\end{remark}

Consider the composition
$$B \xrightarrow{f} \Gr_n(\R^q) \xrightarrow{i} \Gr_{n+1}(\R^{q+1}).$$
Since $(i\circ f)^*(R)=0$, there is a class $r\in H^{n+1}(B\xrightarrow{i\circ f}\Gr_{n+1}(\R^{q+1}))$ such that $r\vert_{\Gr_{n+1}(\R^{q+1})}=R$. We call such a class \(r\) \textit{a relative lift of the relation $R$}.

Let $(X,g_X)$ be an $R$-manifold with $X$ bounding (for brevity, we call \((X,g_X)\) a \textit{bounding \(R\)-manifold}). Choose a null-cobordism $Y$ for $X$. Then we have the commutative diagram
\[\begin{tikzcd}
	B \\
	& X & Y \\
	{\Gr_n(\R^q)} &&& {\Gr_{n+1}(\R^{q+1}).}
	\arrow["f"', from=1-1, to=3-1]
	\arrow["{g_X}"', from=2-2, to=1-1]
	\arrow[hook, from=2-2, to=2-3]
	\arrow["{\tau_X}", from=2-2, to=3-1]
	\arrow["{\tau_Y}"', from=2-3, to=3-4]
	\arrow["i"', from=3-1, to=3-4]
\end{tikzcd}\]
Consider the morphism of pairs $(g_X,\tau_Y) \colon (X\hookrightarrow Y)\to (B\xrightarrow{i\circ f}\Gr_{n+1}(\R^{q+1}))$.
\begin{definition}\label{def:generalized_semicharacteristic}
For a bounding $R$-manifold $(X,g_X)$, the \textit{secondary Stiefel--Whitney number} is given by the formula
$$\varkappa_r^Y(X,g_X):=\langle(g_X,\tau_Y)^*(r),[Y]\rangle\in\Z/2.$$
\end{definition}
The following lemma shows that the secondary Stiefel--Whitney number does not depend on the choice of a null-cobordism.
\begin{lemma}\label{ind}
Let $(X,g_X)$ be a closed $R$-manifold and $Y_1,Y_2$ be null-cobordisms for $X$. Then we have $\varkappa_r^{Y_1}(X,g_X)=\varkappa_r^{Y_2}(X,g_X)$.
\end{lemma}
\begin{proof}
Choose a compact submanifold \(A\) of \(\R_{+}^{q+1}:=\{(x_1,\dots,x_{q+1}) \in \R^{q+1} \mid x_{q+1}\geq 0\}\) such that \(X\) bounds \(A\). Consider the smooth submanifolds \(Z_i \defeq Y_i \cup A\) of \(\R^{q+1}=\R_+^{q+1}\cup \R_-^{q+1}\). Then \(\tau_{Y_i} = \tau_{Z_i} \circ (Y_i \hookrightarrow Z_i)\) for \(i = 1,2\). In addition, there is the commutative diagram for \(i = 1,2\):
\[\begin{tikzcd}
	&& {H^{n+1}(B\xrightarrow{i\circ f}\Gr_{n+1}(\R^{q+1}))} \\
	\\
	{H^{n+1}(X\hookrightarrow Y_i)} && {H^{n+1}(X\hookrightarrow Z_i)} && {H^{n+1}(X\hookrightarrow A).} \\
	&& {H^{n}(X)}
	\arrow["{{{(g_X, \tau_{Y_i})^*}}}"', from=1-3, to=3-1]
	\arrow["{{{(g_X,\tau_{Z_i})^*}}}", from=1-3, to=3-3]
	\arrow["{{{(g_X,\tau_A)^*}}}", from=1-3, to=3-5]
	\arrow["{(\id_X,Y_i\hookrightarrow Z_i)^*}"', from=3-3, to=3-1]
	\arrow["{(\id_X,A\hookrightarrow Z_i)^*}", from=3-3, to=3-5]
	\arrow["\delta", from=4-3, to=3-1]
	\arrow["\delta"', from=4-3, to=3-3]
	\arrow["\delta"', from=4-3, to=3-5]
\end{tikzcd}\]

From the commutativity of the upper left triangle we conclude that   $$\varkappa_r^{Y_i}(X,g_X)=\langle (g_X, \tau_{Y_i})^*(r),[Y_i]\rangle=\langle (\id_X, Y_i \hookrightarrow Z_i)^* \circ (g_X, \tau_{Z_i})^*(r),[Y_i]\rangle$$
for $i=1,2$. 

For the class $(g_X, \tau_{Z_i})^*(r)\in H^{n+1}(X\hookrightarrow Z_i)$ we have $$(g_X, \tau_{Z_i})^*(r)\vert_{Z_i} =\tau_{Z_i}^*(r\vert_{\Gr_{n+1}(\R^{q+1})})=\tau_{Z_i}^*(R).$$
Since $Z_i$ is closed and $R$ is a relation, we have $(g_X, \tau_{Z_i})^*(r)\vert_{Z_i}=0$. Then, from the long exact sequence of the map $(X\hookrightarrow Z_i)$ we conclude that $(g_X, \tau_{Z_i})^*(r)=\delta a_i$ for some $a_i\in H^{n}(X)$.

Thus
\begin{align*}
\varkappa_r^{Y_i}(X,g_X) &= \langle (\id_X, Y_i \hookrightarrow Z_i)^* \circ (g_X, \tau_{Z_i})^*(r), [Y_i]\rangle \\
&= \langle (\id_X, Y_i \hookrightarrow Z_i)^* (\delta a_i),[Y_i]\rangle = \langle \delta a_i,[Y_i]\rangle=\langle a_i, [X]\rangle.
\end{align*}
On the other hand
\begin{align*}
\langle (g_X, \tau_{A})^*(r),[A]\rangle &= \langle (\id_X, A\hookrightarrow Z_i)^* \circ (g_X, \tau_{Z_i})^* (r), [A]\rangle \\
&= \langle (\id_X, A\hookrightarrow Z_i)^* (\delta a_i),[A]\rangle = \langle \delta a_i,[A]\rangle=\langle a_i, [X]\rangle.
\end{align*}
So we have $\varkappa_r^{Y_i}(X,g_X)=\langle (g_X, \tau_{A})^*(r),[A]\rangle$ for $i=1,2$ and hence $\varkappa_r^{Y_1}(X,g_X)=\varkappa_r^{Y_2}(X,g_X)$.
\end{proof}

Since $\varkappa_r^Y$ does not depend on $Y$, we write $\varkappa_r$.
\begin{remark} Two \((B,f)\)-manifolds \((X,g_X)\) and \((X',g_{X'})\) are  called \((B,f)\)-\textit{isotopic} if there exists a continuous family of \((B,f)\)-manifolds \((X_t,g_{X_t})\), \(t\in [0,1]\), such that \((X_0,g_{X_0})=(X,g_X)\) and \((X_1,g_{X_1})=(X',g_{X'})\). The \((B,f)\)-isotopy relation is clearly an equivalence relation. The secondary Stiefel--Whitney number \(\varkappa_R\) is invariant under this relation.
\end{remark}
\begin{proposition}\label{prop:lift_independence}
If $R\neq w_{n+1}, 0$, then the relative lift $r$ of relation $R$ is unique.
\end{proposition}
\begin{proof}
Consider the cohomological Serre spectral sequence for the fibration sequence
$$K(\Z/2, n) \to B\xrightarrow{f} \Gr_n(\R^q).$$
Let $\iota\in H^{n}(K(\Z/2,n))$ be the generator. Then $d_{n+1}(\iota)=i^*(R)$. Since
$$\mathrm{ker}(i^*)= \Z/2 = \{0,w_{n+1}\},$$
one has $i^*(R)\neq 0$, and the differential
$$d_{n+1}\colon E_{n+1}^{0,n} = H^0(\Gr_n(\R^q);H^n(K(\Z/2,n)))\to H^{n+1}(\Gr_n(\R^q);H^0(K(\Z/2,n))) = E_{n+1}^{n+1,0}$$ is injective, and hence ${E}_{n+2}^{0,n} = 0$. Thus, the map $f^*\colon H^n(\Gr_n(\R^q))\to H^n(B)$ is an isomorphism. 
Consider the long exact sequence of the map \((B\xrightarrow{i\circ f} \Gr_{n+1}(\R^{q+1}))\):
\[\begin{tikzcd}
	{H^n(\Gr_{n+1}(\R^{q+1}))} & {H^n(B)} & {H^{n+1}(B\to\Gr_{n+1}(\R^{q+1}))} & {H^{n+1}(\Gr_{n+1}(\R^{q+1}))}
	\arrow["{f^*\circ i^*}", from=1-1, to=1-2]
	\arrow["\delta", from=1-2, to=1-3]
	\arrow["{\vert_{\Gr_{n+1}(\R^{q+1})}}", from=1-3, to=1-4]
\end{tikzcd}\]
The composition \(f^*\circ i^*\) is an isomorphism. So, by exactness, we conclude that $$\vert_{\Gr_{n+1}(\R^{q+1})}\colon H^{n+1}(B\xrightarrow{i\circ f}\Gr_{n+1}(\R^{q+1}))\to H^{n+1}(\Gr_{n+1}(\R^{q+1}))$$ is a monomorphism.

\end{proof}
Let \(R \in \mathcal{R}_{n+1}\) be some relation and \(c \colon \Gr_n(\R^q) \to K(\Z/2, n+1)\) be a classifying map of the cohomological class \(i^*(R) \in H^{n+1}(\Gr_n(\R^q))\). Consider the following commutative square:
\[\begin{tikzcd}
	B & {PK(\Z/2,n+1)} \\
	{\Gr_{n+1}(\R^{q+1})} & {K(\Z/2,n+1).}
	\arrow[from=1-1, to=1-2]
	\arrow["{i\circ f}"', from=1-1, to=2-1]
	\arrow[from=1-2, to=2-2]
	\arrow["c"', from=2-1, to=2-2]
\end{tikzcd}\]
Denote by $\tilde\iota$ the generator of the group $H^{n+1}(PK(\Z/2,n+1) \to K(\Z/2,n+1))$.
Then there is a canonical relative lift $r \in H^{n+1}(B \to \Gr_{n+1}(\R^{q+1}))$ of $R$ defined as an inverse image of $\tilde\iota$ under the morphism of pairs induced by the square diagram above.
Next, we write $\varkappa_R$ to denote $\varkappa_r$ for this canonical lift $r$. By Proposition~\ref{prop:lift_independence}, if \(R\neq w_{n+1},0\), the canonical lift \(r\) is the only relative lift that exists.
\begin{lemma}\label{lem:choice_of_lift}
 Suppose that \(X\) and \(Y\) are two spaces homotopy equivalent to CW complexes, and \(f\colon X\to Y\) and \(c\colon Y\to K(\Z/2,i)\) are two continuous maps. For every class \(x\in H^{i}(X\xrightarrow{f}Y)\) such that \(x\vert_Y = c^*(\iota)\in H^{i}(Y)\) there exists a map \(p\colon X\to PK(\Z/2,i)\) such that the following diagram commutes 
\[\begin{tikzcd}
	{X} & {PK(\Z/2,i)} \\
	{Y} & {K(\Z/2,i),}
	\arrow["{{{p}}}", dashed, from=1-1, to=1-2]
	\arrow["{f}"', from=1-1, to=2-1]
	\arrow[from=1-2, to=2-2]
	\arrow["c", from=2-1, to=2-2]
\end{tikzcd}\] and
$$(p,c)^*(\tilde\iota) = x \in  H^{i}(X\xrightarrow{f} Y).$$
\end{lemma}
\begin{proof}
    Since \((c\circ f)^*(\iota)=0\), the map \(c\circ f\) is homotopic to the zero map, and there is a map \(\tilde p\colon\mathrm{hocofib}(f)\to K(\Z/2,i)\) such that \(c=\tilde p \circ (Y\to\mathrm{hocofib}(f))\). There is a corresponding map \(p\colon X\to PK(\Z/2,i)\) such that \(c\circ f = (PK(\Z/2,i)\to K(\Z/2,i))\circ p\). It is easy to check that \((p,c)^*(\tilde\iota) = x \in  H^{i}(X\xrightarrow{f} Y)\).
\end{proof}

Let us now consider examples of the secondary Stiefel--Whitney numbers for particular relations.

If $n$ is even, then there is a relation $R=w_{n+1}\in\mathcal{R}_{n+1}$.

According to \cite{Ker}, there is \textit{ a relative Stiefel--Whitney class} $$\tilde{w}_{n+1}\in H^{n+1}(\Gr_{n+1}(\R^{q+1}),\Gr_n(\R^q))\cong H^{n+1}(\Gr_n(\R^q)\xrightarrow{i} \Gr_{n+1}(\R^{q+1})),$$ which has the following properties:
\begin{enumerate}
\item $\tilde{w}_{n+1}\vert_{\Gr_{n+1}(\R^{q+1})}=w_{n+1}$,
\item for an $(n+1)$-dimensional cobordism $Y$ between manifolds $X$ and $X'$ there is an equality
$$\langle(\tau_{X\sqcup X'}, \tau_Y)^*(\tilde{w}_{n+1}),[Y]\rangle=\chi(Y)-\chi(X)\mod 2,$$  where $\chi(-)$ is the Euler characteristic. 
\end{enumerate}
Since $H^{n+1}(\Gr_n(\R^q) \to \Gr_{n+1}(\R^{q+1})) = \Z/2$, and $\tilde{w}_{n+1} \ne 0$, we have $$(f, \id_{\Gr_{n+1}(\R^{q+1})})^*(\tilde{w}_{n+1}) = r,$$
where \(r\) is the canonical lift for $R = w_{n+1}$ defined above.

\begin{proposition}
Let $n$ be an even integer.
Then for every bounding $w_{n+1}$-manifold $(X, g_X)$ there is an equality
$$\varkappa_{w_{n+1}}(X,g_X) =\frac{\chi(X)}2\;\mathrm{mod}\; 2.$$
\end{proposition}
\begin{proof}
If the $n$-dimensional manifold $X$ bounds $Y$, then $\chi(X)=2\chi(Y)$. The result follows from the second property of $\tilde{w}_{n+1}$ above.
\end{proof}
Let $n = 2k-1$ be an odd integer. Consider the class $w_{n+1} + v_k^2\in H^{n+1}(\Gr_{n+1}(\R^{q+1}))$, where $v_k$ is the $k$-th Wu class of the vector bundle $\gamma_{n+1}$. By the Wu formula, we have $$w_{n+1} + v_k^2 =\smashoperator[r]{\sum_{i=0}^{k}}\mathrm{Sq}^i(v_{n+1-i}) + v_k^2 = \smashoperator[r]{\sum_{i=0}^{k-1}}\mathrm{Sq}^i(v_{n+1-i}).$$
If $Z$ is a closed $(n+1)$-manifold, then $v_j(\tau_Z)=0$ for $j>k$ and therefore
$$\tau_Z^*(w_{n+1} + v_k^2) = \tau_Z^*\left(\smashoperator[r]{\sum_{i=0}^{k-1}}\mathrm{Sq}^i(v_{n+1-i})\right) = 0.$$
Hence, $w_{n+1} + v_k^2\in\mathcal{R}_{n+1}$.

For a closed odd-dimensional manifold $X$, there is the Kervaire semi-characteristic 
$$K(X):=\frac{\operatorname{dim}H^*(X)}{2}\bmod2.$$

Consider the space $\tilde B = \hofib(\tilde c)$, where $\tilde{c}\colon \Gr_n(\R^q)\to K(\Z/2,k)$
is the classifying map of the Wu class $i^*(v_k)$. Denote by \(\tilde{f} \colon \tilde{B} \to \Gr_n(\R^q)\) the canonical map. The \((\tilde{B},\tilde{f})\)-structure is called \textit{Wu structure} or \(v_k\)-\textit{structure}. For a closed manifold \(X\) for \(k > \mathrm{dim}\,X/2\) Wu's class \(v_k\) automatically vanishes on \(X\), and hence \(v_k\)-structure always exists.

Note that the cohomological class $\tilde{f}^*(v_k)\in H^k(\tilde{B})$ vanishes. From the long exact sequence for the map $\tilde{f}$, there is a class $$u\in H^{k}(\tilde{B}\xrightarrow{i \circ \tilde{f}}\Gr_{n+1}(\R^{q+1}))$$
such that $u\vert_{\Gr_{n+1}(\R^{q+1})}=v_k$.
By \cite[Lemma 4.2]{Pod1}, for a bounding $(\tilde{B},\tilde{f})$-manifold $(X,\tilde{g}_X)$ and a null-cobordism $Y$ for $X$ we have
\begin{equation}
\label{eq:e1}
K(X) = \langle(\tilde{g}_X, \tau_Y)^*(u^2+(\tilde{f}, \mathrm{id}_{\Gr_{n+1}(\R^{q+1})})^*(\tilde{w}_{n+1})),[Y]\rangle.
\end{equation}

As before, $B$ is the homotopy fiber of the classifying map $c\colon \Gr_n(\R^q)\to K(\Z/2,n+1)$ of \(i^*(R)\) for the relation $R = w_{n+1} + v_k^2$ and $f\colon B\to \Gr_n(\R^q)$ is the canonical map. 

Let \(\tilde g_X\) be some \(v_k\)-structure on \(X\). Since $B=\hofib(c)$ and
$$\tilde{f}^*(w_{n+1}+v_k^2)=0 + 0\cdot 0 = 0\in H^{n+1}(\tilde{B}),$$
the set of lifts \(\{p\colon\tilde B\to B) \mid f \circ p = \tilde f\}\) is non-empty. Every choice of lift \(p\) induces an \(R\)-structure on \(X\) given by \(g_X := p \circ \tilde g_X\). 

Now let us show that our invariant $\varkappa_R$ built from the relation $R=w_{n+1}+v_k^2$ is equal to the Kervaire semi-characteristic.
\begin{proposition}\label{th:our_is_kervaire}
Let $n = 2k-1$ be an odd integer, and \(R=w_{n+1}+v_k^2\).
There exists a lift \(p\colon\tilde B\to B\) such that for every bounding manifold $X$ with a $v_k$-structure \(\tilde g_X\) for the corresponding \(R\)-manifold \((X,p\circ\tilde g_X)\) the equality 
$\varkappa_R(X, p\circ\tilde g_X) = K(X)$ holds.
\end{proposition}
\begin{proof}
Since \(u^2 + (\tilde f, \id)^*(\tilde w_{n+1})\in H^{n+1}(\tilde{B}\xrightarrow{i\circ \tilde{f}}\Gr_{n+1}(\R^{q+1}))\) lifts \(w_{n+1}+v_k^2 =c^*(\iota)\in H^{n+1}(\Gr_{n+1}(\R^{q+1}))\), according to \ref{lem:choice_of_lift}, there is a map \(\tilde p\colon\tilde B\to PK(\Z/2,n+1)\) such that the 
$$(\tilde p, c)^*(\tilde\iota) = u^2 + (\tilde f, \id)^*(\tilde w_{n+1}) \in  H^{n+1}(\tilde{B}\xrightarrow{i\circ \tilde{f}}\Gr_{n+1}(\R^{q+1})).$$
 Since $B=\hofib(c)$, there is a corresponding map \(p\colon\tilde B\to B\) such that $$(B\to PK(\Z/2,n+1))\circ p=\tilde p.$$
 Then we have the equality 
\begin{equation}
\label{eq:e4}
u^2 + (\tilde{f}, \mathrm{id})^*(\tilde{w}_{n+1}) = (p, \id)^*(r)\in H^{n+1}(\tilde B\to\Gr_{n+1}(\R^{q+1})).
\end{equation}
For a bounding \(v_k\)-manifold define \(g_X:= p\circ\tilde g_X\). For the bounding $R$-manifold $(X,g_X)$ with a specified null-cobordism $Y$ there is the equality
\begin{equation}
\label{eq:e2}
\varkappa_R(X,g_X)=\langle(g_X, \tau_Y)^*(r),[Y]\rangle.
\end{equation}
Let \(Y\) be some null-cobordism for a \((\tilde{B}, \tilde f)\)-manifold \((X, \tilde g_X)\).
Consider the commutative diagram 
\begin{equation}\label{eq:e3}
\begin{tikzcd}[cramped]
	&& {H^{n+1}(B\xrightarrow{i\circ f}\Gr_{n+1}(\R^{q+1}))} \\
	{H^{n+1}(X\hookrightarrow Y)} \\
	&& {H^{n+1}(\tilde{B}\xrightarrow{i\circ \tilde{f}}\Gr_{n+1}(\R^{q+1})).}
	\arrow["{{{(g_X, \tau_Y)^*}}}"', curve={height=12pt}, from=1-3, to=2-1]
	\arrow["{{{(p, \id)^*}}}"', from=1-3, to=3-3]
	\arrow["{{{(\tilde{g}_X, \tau_Y)^*}}}", curve={height=-12pt}, from=3-3, to=2-1]
\end{tikzcd}
\end{equation}


Therefore, we have
\begin{align*}
K(X) - \varkappa_R(X,g_X) &\overset{\eqref{eq:e1}, \eqref{eq:e2}}{=} \langle(\tilde{g}_X, \tau_Y)^*(u^2+(\tilde{f}, \mathrm{id})^*(\tilde{w}_{n+1})) - (g_X, \tau_Y)^*(r),[Y]\rangle \\
&\overset{\eqref{eq:e3}}{=} \langle(\tilde{g}_X, \tau_Y)^*(u^2 + (\tilde{f}, \mathrm{id})^*(\tilde{w}_{n+1}) - (p, \id)^*(r)), [Y]\rangle\\
&\overset{\eqref{eq:e4}}{=} 0.\qedhere
\end{align*}

\end{proof}

%% file: cob.tex
\section{Cobordism group $\Omega_n^R$}
\label{sec:cobordism_group}
Fix a relation $R\in \mathcal{R}_{n+1}$ and consider its classifying map $\Gr_{n+1}(\R^{q+1})\xrightarrow{c_1} K(\Z/2,n+1)$. Let $B_1$ be the pullback of the path space over \(K(\Z/2,n+1)\) along $c_1$ and $f_1 \colon B_1\to \Gr_{n+1}(\R^{q+1})$ be the canonical map.

Consider the pullback square
\[\begin{tikzcd}[cramped]
B & {B_1} \\
{\Gr_n(\R^q)} & {\Gr_{n+1}(\R^{q+1}).}
\arrow["j", from=1-1, to=1-2]
\arrow["f"', from=1-1, to=2-1]
\arrow["\lrcorner"{anchor=center, pos=0.125}, draw=none, from=1-1, to=2-2]
\arrow["{{f_1}}", from=1-2, to=2-2]
\arrow["i"', from=2-1, to=2-2]
\end{tikzcd}\]

Then we have the morphism of pairs $(j,\id)\colon(B\xrightarrow{i\circ f} \Gr_{n+1}(\R^{q+1}))\to (B_1\xrightarrow{f_1} \Gr_{n+1}(\R^{q+1}))$ and the corresponding morphism of long exact sequences:
\[\begin{tikzcd}
	{H^{n+1}(B_1\xrightarrow{f_1}\Gr_{n+1}(\R^{q+1}))} && {H^{n+1}(\Gr_{n+1}(\R^{q+1}))} && {H^{n+1}(B_1)} \\
	{H^{n+1}(B\xrightarrow{i\circ f}\Gr_{n+1}(\R^{q+1}))} && {H^{n+1}(\Gr_{n+1}(\R^{q+1}))} && {H^{n+1}(B).}
	\arrow["{\vert_{\Gr_{n+1}(\R^{q+1})}}", from=1-1, to=1-3]
	\arrow["{(j,\id)^*}"', from=1-1, to=2-1]
	\arrow["{f_{1}^*}", from=1-3, to=1-5]
	\arrow[equals, from=1-3, to=2-3]
	\arrow["{j^*}"', from=1-5, to=2-5]
	\arrow["{\vert_{\Gr_{n+1}(\R^{q+1})}}", from=2-1, to=2-3]
	\arrow["{f^*\circ i^*}", from=2-3, to=2-5]
\end{tikzcd}\]

There is a canonical map $p\colon B_1 \to PK(\Z/2, n+1)$:
\[\begin{tikzcd}[cramped]
&& {PK(\Z/2,n+1)} \\
{B_1} & {\Gr_{n+1}(\R^{q+1})} & {K(\Z/2,n+1)}
\arrow[from=1-3, to=2-3]
\arrow["p", dashed, from=2-1, to=1-3]
\arrow["{f_1}"', from=2-1, to=2-2]
\arrow["{c_1}"', from=2-2, to=2-3]
\end{tikzcd}\]
As before, denote by $\tilde\iota$ the generator of the group $H^{n+1}(PK(\Z/2,n+1) \to K(\Z/2,n+1))$.
Define a cohomological class $s \in H^{n+1}(B_1 \to \Gr_{n+1}(\R^{q+1}))$ as an inverse image $s = (p, c_1)^*(\tilde\iota)$. This class lifts our relation: $s\vert_{\Gr_{n+1}(\R^{q+1})}=R$.

This allows us to define a relative lift
\begin{equation}
\label{eq:def_of_can_r}
r:=(j,\id)^*(s) \in H^{n+1}(B\xrightarrow{i\circ f}\Gr_{n+1}(\R^{q+1}))
\end{equation}
for $R$, and hence the secondary Stiefel--Whitney number $\varkappa_R$ is defined for a bounding $(B,f)$-manifold.
\begin{remark}
Note that the secondary Stiefel--Whitney number is additive: for two bounding $R$-manifolds $(X,g_X)$ and $(X',g_{X'})$ one can choose two disjoint null-cobordisms $Y$ and $Y'$ for ${X}$ and ${X}'$ respectively and then the following sequence of equalities hold
\begin{align*}
\varkappa_R(X\sqcup X',g_{X\sqcup X'}) &= \langle({g}_{{X}}\sqcup {g}_{{X}'}, \tau_{Y}\sqcup \tau_{Y'})^*(r),[Y\sqcup Y'] \rangle \\
&= \langle({g}_{{X}}, \tau_{Y})^*(r),[Y] \rangle + \langle({g}_{{X}'}, \tau_{Y'})^*(r),[Y'] \rangle =\varkappa_R({X},g_{{X}})+\varkappa_R({X}',g_{{X}'}).
\end{align*}
\end{remark}

\begin{definition}\label{def:R_cobordism_group}
Let $\Omega_n^R$ be the abelian group of $(B_1,f_1)$-cobordism classes of $R$-manifolds. For an $R$-manifold $(X,g_X)$ denote its class as $[X,g_X]^R$.
\end{definition}
Consider the natural forgetting homomorphism $U\colon\Omega_n^R \to \Omega_n^O$ given by $U([X,g_X]^R)=[X]^O$. By Remark~\ref{exi}, this homomorphism is surjective. The next lemma effectively describes its kernel.

\begin{lemma}\label{lem:kernel_characterization}
For a bounding $R$-manifold $(X,g_X)$, $[X,g_X]^R=0$ if and only if $\varkappa_R(X,g_X)=0$.
\end{lemma}
\begin{proof}
Suppose $[X,g_X]^R=0$. Then there is a $(B_1,f_1)$-null-cobordism $(Y,g_Y)$:
\[\begin{tikzcd}
	B &&& {B_1} \\
	& X & Y \\
	{\Gr_n(\R^q)} &&& {\Gr_{n+1}(\R^{q+1}).}
	\arrow["j", from=1-1, to=1-4]
	\arrow["f"', from=1-1, to=3-1]
	\arrow["{f_1}", from=1-4, to=3-4]
	\arrow["{g_{X}}"', from=2-2, to=1-1]
	\arrow[hook, from=2-2, to=2-3]
	\arrow["{\tau_X}"', from=2-2, to=3-1]
	\arrow["{g_Y}", from=2-3, to=1-4]
	\arrow["{\tau_Y}", from=2-3, to=3-4]
	\arrow["i", from=3-1, to=3-4]
\end{tikzcd}\]
We have  $\varkappa_R(X,g_X) = \langle(g_X,\tau_Y)^*(r),[Y]\rangle$.

There is the commutative square 
\[\begin{tikzcd}
	{H^{n+1}(B_1\xrightarrow{f_1} \Gr_{n+1}(\R^{q+1}))} && {H^{n+1}(B_1\xrightarrow{\id} B_1)} \\
	{H^{n+1}(B\xrightarrow{i\circ f} \Gr_{n+1}(\R^{q+1}))} && {H^{n+1}(X\hookrightarrow Y).}
	\arrow["{(\id,f_1)^*}", from=1-1, to=1-3]
	\arrow["{(j,\id)^*}"', from=1-1, to=2-1]
	\arrow["{(j\circ g_X,g_Y)^*}", from=1-3, to=2-3]
	\arrow["{(g_X,\tau_Y)^*}", from=2-1, to=2-3]
\end{tikzcd}\]
 Note that \(H^{n+1}(B_1\xrightarrow{\id}B_1)=0\). Then we have $$\langle(g_X,\tau_Y)^*(r),[Y]\rangle\overset{\eqref{eq:def_of_can_r}}{=}\langle(g_X,\tau_Y)^{*}\circ (j,\id)^{*}(s),[Y]\rangle=0.$$


Conversely, suppose $\varkappa_R(X,g_X)=0$. Choose a connected\footnote{This is only possible if $n \ne 0$. For zero-dimensional manifolds, this will be considered separately in the example \ref{ex}.} null-cobordism $Y$ for $X$. We need to find a map $g_Y\colon Y\to B_1$  such that $\tau_Y=f_1\circ g_Y$ and $g_Y\vert_{X}=j \circ g_X$. Since $B_1=\mathrm{hofib}(c_1)$, the existence of such a map is equivalent to the existence of the map $h_Y\colon Y\to PK(\Z/2,n+1)$ such that the following diagram commutes:
\[\begin{tikzcd}
	X && {PK(\Z/2,n+1)} \\
	Y && {K(\Z/2,n+1).}
	\arrow["{p\circ j \circ g_X}", from=1-1, to=1-3]
	\arrow[hook, from=1-1, to=2-1]
	\arrow[from=1-3, to=2-3]
	\arrow["{h_Y}", dashed, from=2-1, to=1-3]
	\arrow["{c_1\circ\tau_Y}", from=2-1, to=2-3]
\end{tikzcd}\]

The obstruction to the existence of $h_Y$ is the class $(p\circ j\circ g_X, c_1\circ \tau_Y)^*(\tilde\iota)$.
Note that
$$(p\circ j\circ g_X, c_1\circ \tau_Y)^*(\tilde\iota)= (j\circ g_X,\tau_Y)^*(s) = \left((g_X,\tau_Y)^* \circ (j,\id)^*\right)(s) =(g_X,\tau_Y)^*(r).$$
By assumption,
$$0=\varkappa_R(X,g_X)=\langle(g_X,\tau_Y)^*(r),[Y]\rangle.$$
Since $[Y]$ is the generator of $H_{n+1}(X\hookrightarrow Y)$, we have $(g_X,\tau_Y)^*(r)=0$. Hence, there is some $(B_1,f_1)$-structure $g_Y$ on the null-cobordism $Y$ compatible with $g_X$. So we have $[X,g_X]^R=0$.
\end{proof}

\begin{proposition}\label{not0}
    For a bounding non-empty manifold \(X\) there is an \(R\)-structure \(g_X\) such that \(\varkappa_R(X,g_X)\neq 0\).
\end{proposition}
\begin{proof}
    Choose null-cobordism \(Y\) for \(X\) and consider the following diagram 
\[\begin{tikzcd}
	X & B & {PK(\Z/2,n+1)} \\
	Y & {\Gr_{n+1}(\R^{q+1})} & {K(\Z/2,n+1).}
	\arrow["{g_X}", dashed, from=1-1, to=1-2]
	\arrow["{p}", curve={height=-18pt}, dashed, from=1-1, to=1-3]
	\arrow[hook, from=1-1, to=2-1]
	\arrow[from=1-2, to=1-3]
	\arrow["{i\circ f}"', from=1-2, to=2-2]
	\arrow[from=1-3, to=2-3]
	\arrow["{\tau_Y}", from=2-1, to=2-2]
	\arrow["c_1", from=2-2, to=2-3]
\end{tikzcd}\]
The existence of a declared \(R\)-structure \(g_X\) is equivalent to the existence of map \(p\) such that outer diagram commutes and \((p,c_1\circ\tau_Y)^*(\iota)\neq 0\in H^{n+1}(X\hookrightarrow Y)\). According to \ref{lem:choice_of_lift} such a map \(p\) exists.
\end{proof}

\begin{corollary}\label{cor:cobordism_group_descr}
There is a short exact sequence
\begin{equation}\label{eq:cobordism_ses}
0 \to \mathrm{ker}(U) \to \Omega^R_n \xrightarrow{U} \Omega^O_n \to 0,
\end{equation}
and $\varkappa_R\colon\mathrm{ker}(U)\to \Z/2$ is an isomorphism.
\end{corollary}
\begin{proof}
According to Remark~\ref{exi}, the homomorphism $U$ is surjective. By Lemma~\ref{lem:kernel_characterization}, \(\varkappa_R\) is injective. By Proposition~\ref{not0}, it is surjective as well.
\end{proof}

\begin{example}
Consider the relation $w_1^2+w_2\in\mathcal{R}_2$ and the corresponding group $\Omega_1^{w_1^2+w_2}$.
\begin{itemize}
\item According to \cite{Kir}, the obstruction to the existence of the $\mathrm{Pin^-}$-structure is exactly the class $w_{1}^2 + w_2$. So,  $\Omega_1^{w_1^2+w_2}\cong \Omega^{\mathrm{Pin}^-}_1$.
\item Since \(\Omega^O_1 = 0\), we have \(\ker(U) = \Omega^{w_1^2 + w_2}\). We conclude that $\varkappa_{w_{1}^2 + w_2} \colon \Omega^{\mathrm{Pin}^-}_1 \to \Z/2$ is an isomorphism by the previous corollary.
\item Since $w_1 = v_1$, by Theorem~\ref{th:our_is_kervaire}, the orientation \(\tilde g_X\) on the 1-dimensional manifold $X$ gives rise to a $\mathrm{Pin}^-$-structure $g_X$ such that $\varkappa_{w_{1}^2 + w_2}(X,g_X) = K(X)$.
\end{itemize}
\end{example}
\begin{example}
Consider the relation $w_4+v_2^2\in\mathcal{R}_4$ and the corresponding group $\Omega^{w_4+v_2^2}_3$.
\begin{itemize}
\item Since $\Omega_3^O\cong 0$, we have $\ker(U) = \Omega^{w_4 + v_2^2}_3$. As before, this implies that there is an isomorphism \(\varkappa_{w_4 + v_2^2} \colon \Omega_3^{w_4 + v_2^2} \to \Z/2\).
\item Since $v_2=w_2+w_1^2$, according to \cite{Kir} the existence of $(\tilde{B},\tilde{f})$-structure is equivalent to the existence of $\mathrm{Pin}^-$-structure. Again, according to Theorem~\ref{th:our_is_kervaire} for the 3-dimensional $\mathrm{Pin}^-$-manifold $(X,\tilde{g}_X)$ we have a corresponding $(w_4+v_2^2)$-structure $g_X$ such that equality $\varkappa_{w_4 + v_2^2}(X,g_X) = K(X)$ holds.
\end{itemize}
\end{example}
 
\begin{example}\label{ex}
Consider $w_{n+1}\in\mathcal{R}_{n+1}$ for even $n$. For a bounding $w_{n+1}$-manifold $(X,g_X)$ and null-cobordism $Y$ for $X$ there is a sequence of equalities
$$\varkappa_{w_{n+1}}(X,g_X) =\chi(Y)\bmod 2=\tfrac{1}{2}\chi(X)\bmod2.$$
Consider $(\R P^n,g_{\R P^n})$ with any $w_{n+1}$-structure $g_{\R P^n}\colon\R P^n\to B$. Since $$\varkappa_{w_{n+1}}(2\cdot [\R P^n,g_{\R P^n}]^{w_{n+1}})=\chi(\R P^n)\bmod 2=1$$ we have $2 \cdot [\R P^n,g_{\R P^n}]^{w_{n+1}} \neq 0\in\Omega_n^{w_{n+1}}$. The exact sequence \eqref{eq:cobordism_ses} does not split, and there is the homomorphism $\chi\bmod 4\colon\Omega_n^{w_{n+1}}\to \Z/4 $. 

For the case $n=0$ we have $\mathcal{R}_{1} =\{0, w_1\}$, and every \(0\)-dimensional manifold is a disjoint union of some copies of \(\mathbb{R}P^0 \simeq *\). Thus, the calculations above are still valid, and we may conclude that $\Omega_0^{w_1}=\Z/4$. 
\end{example}

%% file: geometric_model.tex
\section{Geometric constructions}
\label{sec:geometric_constructions}
\label{sec:coordinate_approach}
\subsection{Explicit model for (co)chains}
For every topological space \(X\) there are singular chains \(C_*(X) = \Z/2[\mathrm{Hom}_{\mathrm{Top}}(\Delta^*, X)]\) and cochains \(C^*(X) = \mathrm{Hom}_{\mathrm{Ab}}(C_*(X), \Z/2)\).

Consider a morphism of complexes \(f \colon C_* \to D_*\), as a bicomplex, where \(D_i\) is placed in a bidegree \((i,0)\) and \(C_i\) is placed in bidegree \((i,1)\), respectively. Then there is a total complex whose \(i\)-chains are defined as \(\mathrm{Tot}(f)_i = D_i \oplus C_{i-1}\). Analogously, for a morphism of cochain complexes \(f \colon C^* \to D^*\), there is a totalization \(\mathrm{Tot}(f)^i = C^i \oplus D^{i-1}\).

If \(f \colon X \to Y\) is a continuous map, then there are induced morphisms \(f_* \colon C_*(X) \to C_*(Y)\) and \(f^*\colon C^*(Y) \to C^*(X)\). We can define relative chain and cochain complexes as
\begin{align*}
C_*(X \xrightarrow{f} Y) &\defeq \mathrm{Tot}(f_*), \\
C^*(X \xrightarrow{f} Y) &\defeq \mathrm{Tot}(f^*).
\end{align*}
We will denote their elements as
\[c = (c_{(0)}, c_{(1)}) \in C_i(X \xrightarrow{f} Y) = C_i(Y) \oplus C_{i-1}(X)\]
and
\[u = (u^{(0)}, u^{(1)}) \in C^i(X \xrightarrow{f} Y) = C^i(Y) \oplus C^{i-1}(X).\]
One has \(c \in Z_*(X \xrightarrow{f} Y)\) iff \(c_{(1)} \in Z_*(X)\) and \(f_*(c_{(1)}) = \partial_Y(c_{(0)})\); and \(u \in Z^*(X \xrightarrow{f} Y)\) iff \(u^{(0)} \in Z^*(Y)\) and \(f^*(u^{(0)}) = \delta_X(u^{(1)})\).

It is easy to see that there are canonical isomorphisms
\begin{align*}
H_k(C_*(X \to Y)) &= H_k(X \to Y), \\
H^k(C^*(X \to Y)) &= H^k(X \to Y).
\end{align*}

The long exact sequence of cohomology groups arises as the long exact sequence associated with the short exact sequence of cochain complexes
\[0 \to C^{*-1}(X) \to C^*(X \to Y) \to C^*(Y) \to 0.\]

We won't stop at this. For a commutative square of spaces
\[\begin{tikzcd}[cramped]
X & Y \\
{X'} & {Y'}
\arrow["f", from=1-1, to=1-2]
\arrow["a"', from=1-1, to=2-1]
\arrow["b", from=1-2, to=2-2]
\arrow["{f'}"', from=2-1, to=2-2]
\end{tikzcd}\]
we define its complex of cochains as
\[C^*\left(\begin{tikzcd}[cramped,sep=tiny,font=\scriptsize]
         X & Y \\
	X' & Y'
	\arrow[from=1-1, to=1-2]
	\arrow[from=1-1, to=2-1]
	\arrow[from=1-2, to=2-2]
	\arrow[from=2-1, to=2-2]
\end{tikzcd}\right) \defeq \mathrm{Tot}(C^*(X' \to Y') \to C^*(X \to Y))\]
and
\[H^*\left(\begin{tikzcd}[cramped,sep=tiny,font=\scriptsize]
         X & Y \\
	X' & Y'
	\arrow[from=1-1, to=1-2]
	\arrow[from=1-1, to=2-1]
	\arrow[from=1-2, to=2-2]
	\arrow[from=2-1, to=2-2]
\end{tikzcd}\right) \defeq H^*\left(C^*\left(\begin{tikzcd}[cramped,sep=tiny,font=\scriptsize]
         X & Y \\
	X' & Y'
	\arrow[from=1-1, to=1-2]
	\arrow[from=1-1, to=2-1]
	\arrow[from=1-2, to=2-2]
	\arrow[from=2-1, to=2-2]
\end{tikzcd}\right)\right).\]
Observe that
\[C^i\left(\begin{tikzcd}[cramped,sep=tiny,font=\scriptsize]
         X & Y \\
	X' & Y'
	\arrow[from=1-1, to=1-2]
	\arrow[from=1-1, to=2-1]
	\arrow[from=1-2, to=2-2]
	\arrow[from=2-1, to=2-2]
\end{tikzcd}\right) = C^i(Y') \oplus C^{i-1}(X') \oplus C^{i-1}(Y) \oplus C^{i-2}(X).\]
There is an associated long exact sequence
\[\dots \to H^{*-1}(X \to Y) \xrightarrow{\delta} H^*\left(\begin{tikzcd}[cramped,sep=tiny,font=\scriptsize]
         X & Y \\
	X' & Y'
	\arrow[from=1-1, to=1-2]
	\arrow[from=1-1, to=2-1]
	\arrow[from=1-2, to=2-2]
	\arrow[from=2-1, to=2-2]
\end{tikzcd}\right) \xrightarrow{\vert_{f'}} H^*(X' \to Y') \xrightarrow{(a,b)^*} H^*(X \to Y) \to \dots.\]

There is a multiplication map
\[\cdot\colon C^i(X'\xrightarrow{f'} Y')\otimes C^j(Y\xrightarrow{b} Y')\to C^*\left(\begin{tikzcd}[cramped,sep=tiny,font=\scriptsize]
    X & Y \\
    X' & Y'
	\arrow[from=1-1, to=1-2]
	\arrow[from=1-1, to=2-1]
	\arrow[from=1-2, to=2-2]
	\arrow[from=2-1, to=2-2]
\end{tikzcd}\right)\] defined by the equality
$$(u^{(0)},u^{(1)})\cdot(v^{(0)},v^{(1)})=(u^{(0)}v^{(0)},u^{(1)}f'^*(v^{(0)}), b^*(u^{(0)})v^{(1)}, a^*(u^{(1)})f^*(v^{(1)})).$$
The multiplication map is a morphism of chain complexes, hence it is defined on cohomology.

The relative pairing realized on the level of chains via the formula
\begin{align*}
\langle -, - \rangle \colon C^i(X \to Y) \otimes C_i(X \to Y) \to \Z/2, \quad 
\langle u, c \rangle = \langle u^{(0)}, c_{(0)} \rangle_Y + \langle u^{(1)}, c_{(1)} \rangle_X,
\end{align*}
where \(\langle -, - \rangle_X\) and \(\langle -, - \rangle_Y\) denote cochain-chain pairings on \(X\) and \(Y\), respectively.

\subsection{Thom space}

It is convenient to consider the sphere \(S^q\) as the one-point compactification $\R^q \cup \{*\}$. Let \(D^{q+1} \defeq \R^{q+1}_- \cup \{*\}\) be a model for a closed disk. Then there is a canonical inclusion \(S^q \hookrightarrow D^{q+1}\).

For a Euclidean vector bundle \(\xi\), denote by \(D\xi\) the corresponding \textit{disk bundle} and by \(S\xi\) the corresponding \textit{sphere bundle}.
The \textit{Thom space} of a bundle $\xi$ is defined as the quotient space $\Th \xi \defeq D\xi/S\xi$. It has a canonical basepoint \(* = \{S\xi\}\). For a continuous map \(f \colon Y \to X\) and a vector bundle \(\xi\) over \(X\), there is an induced map \(\Th(f) \colon \Th f^*\xi \to \Th\xi\).

Let \(\xi\) be a vector bundle of rank \(k\).
We denote by \(\mu(\xi) \in \widetilde H^k(\Th\xi)\) the \textit{Thom class} of \(\xi\), and by \(\varphi \colon H^*(X) \to \widetilde H^{* + k}(\Th\xi)\) the Thom isomorphism.
Abusing the notation, we will denote the image of \(\mu(\xi)\) under the composition
\[\widetilde H^k(\Th\xi) = H^k(* \to \Th\xi) \xrightarrow[\text{excision}]{\simeq} H^k(S\xi \to D\xi)\]
by the same symbol.

\subsection{Pontryagin--Thom collapse map}
\label{sec:pontryagin-thom_collapse}
Let \(\xi\) be a subbundle of the trivial bundle, then we denote its orthogonal complement by \(\xi^\perp\). Two notable examples are the normal bundle \(\nu_X = (TX)^\perp\) of the submanifold \(X\subset\R^q\) of rank \(q - n\) and \(\bar \gamma_n^q = {(\gamma_n^q)}^\perp\) of rank \(q - n\), where \(\gamma_n^q\) is a tautological bundle over \(\Gr_n(\R^q)\).

If \(X\) is an \(n\)-dimensional submanifold of \(\R^q\), then there is an open tubular neighborhood \(U(X)\), and there is a quotient map
\[\mathrm{collapse} \colon S^q = \R^q \cup \{*\} \to \R^q / (\R^q \setminus U(X)) = \Th\nu_X.\]

If \(X\) is a \((B, f)\)-manifold, then there is a commutative diagram
\[\begin{tikzcd}[cramped]
	& B & {f^* \bar\gamma_n^q} \\
	X & {\Gr_n(\R^q)} & {\bar\gamma_n^q.} \\
	& {\nu_X}
	\arrow["f", from=1-2, to=2-2]
	\arrow[from=1-3, to=1-2]
	\arrow["\lrcorner"{anchor=center, pos=0.125, rotate=-90}, draw=none, from=1-3, to=2-2]
	\arrow[from=1-3, to=2-3]
	\arrow["{{g_X}}", from=2-1, to=1-2]
	\arrow["{{\tau_X}}", from=2-1, to=2-2]
	\arrow[from=2-3, to=2-2]
	\arrow[from=3-2, to=2-1]
	\arrow["\lrcorner"{anchor=center, pos=0.125, rotate=135}, draw=none, from=3-2, to=2-2]
	\arrow[from=3-2, to=2-3]
\end{tikzcd}\]
It defines a morphism of vector bundles \(\nu_X \to f^* \bar \gamma_n^q\) that induces a map \(\Th \nu_X \to \Th f^* \bar \gamma_n^q\).
Then we can define a \textit{\((B,f)\)-collapse map} as the composition
\[\mathfrak{a} \colon S^q \xrightarrow{\mathrm{collapse}} \Th \nu_X \to \Th f^* \bar \gamma_n^q.\]

If \(X \subset \R^q\) bounds \(Y \subset \R^{q+1}_-\), then we can define a relative version of the collapse map. Let \(U(Y)\) be an open tubular neighborhood. Then there is a quotient
\[\mathrm{collapse} \colon D^{q+1} = \R^{q+1}_- \cup \{*\} \to \R^{q+1}_- / U(Y) = \Th \nu_Y.\]

We can define the relative collapse map as the composition
\[\mathfrak{b} \colon D^{q+1} \xrightarrow{\mathrm{collapse}} \Th \nu_Y \to \Th \bar \gamma^{q+1}_{n+1}.\]
There is a pullback diagram:
\[\begin{tikzcd}
	{\Gr_n(\R^q)} & {\bar\gamma_n^q} \\
	{\Gr_{n+1}(\R^{q+1})} & {\bar\gamma_{n+1}^{q+1}.}
	\arrow["i"', from=1-1, to=2-1]
	\arrow[from=1-2, to=1-1]
	\arrow["\lrcorner"{anchor=center, pos=0.125, rotate=-90}, draw=none, from=1-2, to=2-1]
	\arrow[from=1-2, to=2-2]
	\arrow[from=2-2, to=2-1]
\end{tikzcd}\]
If the tubular neighborhoods \(U(X)\) and \(U(Y)\) are chosen compatibly\footnote{See \cite{Pod1} for an explicit construction.} there is the commutative square
\begin{equation}\label{a_b_commutativity}
\begin{tikzcd}[cramped,sep=2.25em]
	{S^q} & {D^{q+1}} \\
	{\Th f^*\bar\gamma_{n}^{q}} & {\Th\bar\gamma_{n+1}^{q+1}.}
	\arrow[from=1-1, to=1-2]
	\arrow["{{\mathfrak{a}}}"', from=1-1, to=2-1]
	\arrow["{{\mathfrak{b}}}", from=1-2, to=2-2]
	\arrow["{{\Th(i\circ f)}}"', from=2-1, to=2-2]
\end{tikzcd}
\end{equation}

\subsection{Relative Thom isomorphism}
\label{sec:relative_thom}
Let \(f \colon X \to Y\) be a continuous map, and \(\xi\) be a vector bundle over \(Y\) of rank \(k\).
Consider the following commutative square:
\[\begin{tikzcd}[cramped]
	{S f^*\xi} && {D f^*\xi} \\
	& {(A)} \\
	{S\xi} && {D\xi.}
	\arrow[hook, from=1-1, to=1-3]
	\arrow[from=1-1, to=3-1]
	\arrow[from=1-3, to=3-3]
	\arrow[hook, from=3-1, to=3-3]
\end{tikzcd}\]
We define a morphism \(\mu \cdot (-) \colon H^*(X \xrightarrow{f} Y) \to H^{*+k}\left(\fbox{A}\right)\) as the composite of the isomorphism \(H^*(X\to Y) \to H^*(Df^*\xi\to D\xi)\) induced by the natural projections and multiplication by the Thom class \(\mu\in H^k(S\xi\to D\xi)\). 

There is a natural transformation of two commutative squares of spaces:
\[\begin{tikzcd}[cramped]
	{*} && {\Th f^* \xi} && {S f^*\xi} && {D f^*\xi} \\
	& {(B)} &&&& {(A)} \\
	{*} && {\Th \xi} && {S\xi} && {D\xi.}
	\arrow[from=1-1, to=1-3]
	\arrow[equals, from=1-1, to=3-1]
	\arrow[""{name=0, anchor=center, inner sep=0}, "{{\Th(f)}}"', from=1-3, to=3-3]
	\arrow[hook, from=1-5, to=1-7]
	\arrow[""{name=1, anchor=center, inner sep=0}, from=1-5, to=3-5]
	\arrow[from=1-7, to=3-7]
	\arrow[from=3-1, to=3-3]
	\arrow[hook, from=3-5, to=3-7]
	\arrow["\Pi"', between={0.2}{0.8}, Rightarrow, from=1, to=0]
\end{tikzcd}\]
We denote its components as
\begin{align*}
& \Pi_1 \colon (S \xi \to D\xi) \to (* \to \Th \xi), \\
& \Pi_2 \colon (S f^* \xi \to D f^* \xi) \to (* \to \Th f^* \xi).
\end{align*}

It induces a commutative diagram
\[\begin{tikzcd}
	\dots & {H^{*-1+k}(* \to \Th f^* \xi)} & {H^{*+k}\left(\fbox{B}\right)} & {H^{*+k}(* \to T\xi)} & \dots \\
	\dots & {H^{*-1+k}(Sf^*\xi \to D f^*\xi)} & {H^{*+k}\left(\fbox{A}\right)} & {H^{*+k}(S\xi \to D\xi)} & \dots \\
	\dots & {H^{*-1}(X)} & {H^*(X \xrightarrow{f} Y)} & {H^*(Y)} & \dots.
	\arrow[from=1-1, to=1-2]
	\arrow["\delta", from=1-2, to=1-3]
	\arrow["{{{\Pi_2^*}}}", from=1-2, to=2-2]
	\arrow[from=1-3, to=1-4]
	\arrow["{{{\Pi^*}}}", from=1-3, to=2-3]
	\arrow[from=1-4, to=1-5]
	\arrow["{{{\Pi_1^*}}}", from=1-4, to=2-4]
	\arrow[from=2-1, to=2-2]
	\arrow["\delta", from=2-2, to=2-3]
	\arrow[shift left=2, draw=none, from=2-2, to=3-2]
	\arrow[from=2-3, to=2-4]
	\arrow[from=2-4, to=2-5]
	\arrow[from=3-1, to=3-2]
	\arrow["\varphi", from=3-2, to=2-2]
	\arrow["\cong"', from=3-2, to=2-2]
	\arrow["\delta"', from=3-2, to=3-3]
	\arrow["{{\mu \cdot (-)}}"', from=3-3, to=2-3]
	\arrow[from=3-3, to=3-4]
	\arrow["\varphi", from=3-4, to=2-4]
	\arrow["\cong"', from=3-4, to=2-4]
	\arrow[from=3-4, to=3-5]
\end{tikzcd}\]

\begin{itemize}
\item Multiplication by the Thom class \(\mu \cdot (-) \colon H^*(X \to Y) \to H^{*+k}\left(\fbox{A}\right)\) is an isomorphism by the 5-lemma.
\item The homomorphism \(\Pi^* \colon H^{*+k}\left(\fbox{B}\right) \to H^{*+k}\left(\fbox{A}\right)\) is an isomorphism by the 5-lemma.
\item Since \(H^*(* \to *) = 0\), the homomorphism \(\vert_{\Th(f)} \colon H^{*+k}\left(\fbox{B}\right) \to H^{*+k}(\Th f^* \xi \to \Th\xi)\) is an isomorphism.
\end{itemize}
The composition
\[\Phi \colon H^*(X \to Y) \xrightarrow{\mu \cdot (-)} H^{*+k}\left(\fbox{A}\right) \xrightarrow{(\Pi^*)^{-1}} H^{*+k}\left(\fbox{B}\right) \xrightarrow{\vert_{\Th(f)}} H^{*+k}(\Th f^*\xi \to \Th \xi)\]
is an isomorphism.
We call this map the \textit{relative Thom isomorphism}. Note that by construction of \(\Phi\) for every class \(x \in H^i(X \to Y)\), one has 
\begin{equation}
\label{thomrel}
\varphi(x\vert_Y) = \Phi(x)\vert_{\Th\xi}.
\end{equation}


\begin{lemma}
\label{55}
\leavevmode
\begin{enumerate}
\item For a \((B, f)\)-manifold \((X, g_X)\) consider the following diagram 
\[\begin{tikzcd}
	{ H^n(B)} & {\widetilde H^q(\Th f^* \bar \gamma_n^q)} \\
	{H^n(X)} & {\widetilde H^q(S^q).}
	\arrow["\varphi", from=1-1, to=1-2]
	\arrow["{{g_X^*}}"', from=1-1, to=2-1]
	\arrow["{{\mathfrak{a}^*}}", from=1-2, to=2-2]
\end{tikzcd}\]
Then we have \(\langle g_X^*(k), [X] \rangle = \langle \mathfrak{a}^*(\varphi(k)), [S^q] \rangle\) for every \(k \in H^n(B)\). 
\item For a bounding \((B, f)\)-manifold \((X, g_X)\) and null-cobordism \(Y\) for \(X\) we have the following diagram
\[\begin{tikzcd}
	&& {H^{n+1}(B \xrightarrow{i \circ f} \Gr_{n+1}(\R^{q+1}))} & {H^{q+1}(\Th f^* \bar \gamma_n^q \xrightarrow{\Th(i \circ f)} \Th \bar \gamma_{n+1}^{q+1})} \\
	{} && {H^{n+1}(X \to Y)} & {H^{q+1}(S^q \to D^{q+1}),}
	\arrow["\Phi", from=1-3, to=1-4]
	\arrow["{{(g_X, \tau_Y)^*}}"', from=1-3, to=2-3]
	\arrow["{{(\mathfrak{a},\mathfrak{b})^*}}", from=1-4, to=2-4]
\end{tikzcd}\]
and
\(\langle (g_X, \tau_Y)^*(k), [Y] \rangle = \langle (\mathfrak{a}, \mathfrak{b})^*(\Phi(k)), [D^{q+1}] \rangle\)
for any \(k \in H^{n+1}(B \xrightarrow{i \circ f} \Gr_{n+1}(\R^{q+1}))\).
\end{enumerate}
\end{lemma}
\begin{proof}
\leavevmode
\begin{enumerate}
\item See \cite[Lemma~16.43]{Swi}.
\item The proof is analogous. \qedhere
\end{enumerate}
\end{proof}

\begin{lemma}\label{56}
Let \(X\) be a topological space, $\xi$, $\xi^\perp$ be vector bundles over $X$ such that \(\xi\oplus\xi^{\perp}\) is trivial, and let $\varphi\colon H^*(X) \to \widetilde H^{*+\mathrm{rk} \,\xi^\perp}(\Th\xi^\perp)$ be the Thom isomorphism. Then for every $x\in H^*(X)$ the following equality holds: $$\mathrm{Sq}(\varphi(v(\xi)x))=\varphi(\mathrm{Sq}(x)).$$
\end{lemma} 
\begin{proof}
Let \(\mu = \mu(\xi^\perp) \in H^{\mathrm{rk}\,\xi^\perp}(\Th \xi^\perp)\) be the Thom class.
We have
$$\mathrm{Sq}(\varphi(v(\xi)x))\overset{(1)}{=} \mathrm{Sq}(\mu v(\xi)x)\overset{(2)}{=} \mathrm{Sq}(\mu)\mathrm{Sq}(v(\xi))\mathrm{Sq}(x)\overset{(3)}{=} \mu w(\xi^\perp) w(\xi) \mathrm{Sq}(x)\overset{(4)}{=}\mu \mathrm{Sq}(x)\overset{(5)}{=}\varphi(\mathrm{Sq}(x)),$$
where
\begin{enumerate}
\item definition of the Thom isomorphism,
\item Cartan formula,
\item definition of Stiefel--Whitney classes via the action of Steenrod squares on the Thom class and definition of Wu classes,
\item Whitney formula,
\item definition of Thom isomorphism. \qedhere
\end{enumerate}
\end{proof}

Let $\varphi\colon H^*(\Gr_{n+1}(\R^{q+1}))\to \widetilde H^{*+q-n}(\Th\bar{\gamma}_{n+1}^{q+1})$ be the Thom isomorphism.
\begin{lemma}\label{57}
For every $x\in H^{n-i+1}(\Gr_{n+1}(\R^{q+1}))$ one has
$$\varphi(\mathrm{Sq}^i(x)+v_i(\gamma_{n+1}^{q+1})x) = \sum_{j=0}^{q} \mathrm{Sq}^{q-j+1}(a_j)\in \widetilde H^{q+1}(\Th\bar{\gamma}_{n+1}^{q+1})$$ 
for some $a_j \in \widetilde H^j(\Th \bar \gamma_{n+1}^{q+1})$.
\end{lemma}
\begin{proof}
Denote by $a$ and $b$ the images of $v(\gamma_{n+1}^{q+1})x$ and $\mathrm{Sq}(x)$ under $\varphi$, respectively, and by $a_j, b_j$ their homogeneous components. By Lemma~\ref{56}, \(\Sq(a)=b\). We have
\begin{align*}
\varphi(\mathrm{Sq}^i(x)+v_i(\gamma_{n+1}^{q+1})x) &= \varphi(\mathrm{Sq}^i(x))+\varphi(v_i(\gamma_{n+1}^{q+1})x)=b_{q+1} + a_{q+1} = \\
&= \sum_{j=0}^{q+1} \mathrm{Sq}^{q-j+1}(a_j) + a_{q+1}= \sum_{j=0}^{q} \mathrm{Sq}^{q-j+1}(a_j) + \mathrm{Sq}^0(a_{q+1}) + a_{q+1} = \\
&= \sum_{j=0}^{q} \mathrm{Sq}^{q-j+1}(a_j). \qedhere
\end{align*}
\end{proof}

\begin{remark} \label{dold}
According to \cite{Br},  the subgroup $\mathcal{R}_{n+1}\subset H^{n+1}(\Gr_{n+1}(\R^{q+1}))$ is generated by the set  $$\left\{\mathrm{Sq}^i(x)+v_ix \mid x\in H^{n+1-i}(\Gr_{n+1}(\R^{q+1})),\, i=1,\ldots,n+1\right\}.$$
\end{remark}
Let $\Phi \colon H^*(B\xrightarrow{i\circ f}\Gr_{n+1}(\R^{q+1}))\to H^{*+q-n}(\Th f^*\bar\gamma_n^q\xrightarrow{\Th(i\circ f)} \Th\bar \gamma_{n+1}^{q+1})$ be the relative Thom isomorphism.
\begin{corollary}   \label{59}
Let $R\in\mathcal{R}_{n+1}$ be some relation, and $r\in H^{n+1}(B\xrightarrow{i\circ f}\Gr_{n+1}(\R^{q+1}))$ be its relative lift.
Then
\[\Phi(r)\vert_{\Th\bar\gamma_{n+1}^{q+1}}=\sum_{j=0}^{q}\mathrm{Sq}^{q-j+1}(a_{j})\in \tilde{H}^{q+1}(T\bar{\gamma}_{n+1}^{q+1})\]
for some $a_j\in \tilde{H}^{j}(\Th\bar\gamma_{n+1}^{q+1})$.
\end{corollary}
\begin{proof}
According to Remark~\ref{dold}, 
$$R=\smashoperator[r]{\sum_{i=1}^{n+1}}(\mathrm{Sq}^i(x_i)+v_i(\gamma_{n+1}^{q+1})x_i)$$
for some $x_i\in H^{n-i+1}(\Gr_{n+1}(\R^{q+1}))$.
Then by Lemma~\ref{57} we conclude that 
\begin{align*}
\Phi(r)\vert_{\Th\bar\gamma_{n+1}^{q+1}} \underset{\eqref{thomrel}}{=} \varphi(R) = \sum_{i=1}^{n+1}\varphi(\mathrm{Sq}^i(x_i)+v_i(\gamma_{n+1}^{q+1})x_i) = \sum_{i=1}^{n+1}\smashoperator[r]{\sum_{j=0}^{q}}\mathrm{Sq}^{q-j+1}(a_{ij})  = \sum_{j=0}^q \mathrm{Sq}^{q-j+1}\left(\sum_{i=1}^{n+1}a_{ij}\right).
\end{align*}
Then, we can take \(a_j \defeq \sum_{i = 1}^{n+1} a_{ij}\).
\end{proof}

\subsection{Geometric constructions}

\subsubsection{Representation for image of the relative lift under the Thom isomorphism}
\label{alpha_omega_construction}
Let \(R \in \mathcal{R}_{n+1}\) be some relation and \(r \in H^{n+1}(B \xrightarrow{i \circ f} \Gr_{n+1}(\R^{q+1}))\) be its relative lift. Consider a cocycle \[\rho \in Z^{q+1}(\Th f^* \bar \gamma_n^q \xrightarrow{\Th(i \circ f)} \Th \bar \gamma_{n+1}^{q+1})\] representing \(\Phi(r)\), where \(\Phi\) is the relative Thom isomorphism. Then, \(\rho = (\rho^{(0)}, \rho^{(1)})\), where
\begin{align*}
&\rho^{(0)} \in Z^{q+1}(\Th \bar \gamma_{n+1}^{q+1}), \\
&\rho^{(1)} \in C^q(\Th f^* \bar \gamma_n^q), \\
&\Th(i \circ f)^*(\rho^{(0)}) = \delta(\rho^{(1)}).
\end{align*}
By Corollary~\ref{59}, there are some cocycles \(\alpha_j \in Z^j(\Th \bar \gamma_{n+1}^{q+1})\) and some cochain \(\omega \in C^q(\Th \bar \gamma_{n+1}^{q+1})\)  such that
\begin{equation}
\label{999}
\rho^{(0)} = \rho\vert_{\Th\bar\gamma_{n+1}^{q+1}}=\sum_{j=0}^{q}\mathrm{Sq}^{q-j+1}(\alpha_j)  + \delta(\omega).
\end{equation}

\subsubsection{Representation for the relative fundamental class \([D^{q+1}]\)}
\label{sec:cycle_choice}
For future needs, we have to start working subordinate to some open cover.
Let \(\Sigma\) be an open cover of the sphere \(S^q\). Denote by \(C_*(\Sigma)\) the complex of singular chains in \(S^q\) subordinate to the cover \(\Sigma\) (similar notation will be used for the groups of cocycles and (co)homologies).
There is a linear map \((-)\vert_\Sigma \colon C^*(D^{q+1}) \to C^*(\Sigma)\) given by
the restriction along the composition
\[C_*(\Sigma) \hookrightarrow C_*(S^q) \hookrightarrow C_*(D^{q+1}).\]
Our desire to work locally is justified by the following lemma, which will be applied in the next section.
\begin{lemma}[{\cite[Proposition~5.2]{Pod1}}]
\label{lemma_10000000}
Let \(l \ge 1\) and \(x, y \in Z^{q+1-l}(D^{q+1})\) by such that \(x\vert_\Sigma = y\vert_\Sigma\). Then one has
\[\langle\mathrm{Sq}^l(x), M_{(0)} \rangle = \langle\mathrm{Sq}^l(y), M_{(0)}\rangle.\]
\end{lemma}

Let \(M_{(1)} \in Z_q(\Sigma)\) be a cycle whose image in \(Z_q(S^q)\) represents the fundamental class \([S^q]\). Choose some chain \(M_{(0)} \in C_{q+1}(D^{q+1})\) such that \(\partial M_{(0)} = (S^q \to D^{q+1})_*(M_{(1)})\). Then the fundamental class \([D^{q+1}] \in H_{q+1}(S^q \to D^{q+1})\) is represented by the cycle \[(M_{(0)}, M_{(1)}) \in Z_{q+1}(S^q \to D^{q+1}).\]

Note that since \(M_{(1)} \in Z_q(\Sigma)\), we have
\begin{equation}\label{1001}
\langle \delta \mathfrak{b}^*(\omega), M_{(0)} \rangle = \langle \mathfrak{b}^*(\omega), M_{(1)} \rangle \overset{\eqref{a_b_commutativity}}{=} \langle \mathfrak{a}^*(\Th (i \circ f)^*(\omega))\vert_\Sigma, M_{(1)} \rangle,
\end{equation}
and
\begin{equation}\label{1002}
\langle \mathfrak{a}^*(\rho^{(1)}), M_{(1)} \rangle = \langle \mathfrak{a}^*(\rho^{(1)})\vert_\Sigma, M_{(1)} \rangle.
\end{equation}

\subsubsection{Geometric formula for secondary Stiefel--Whitney number}
\label{sec:geometric_formula_for_semichar}
For a bounding \(R\)-manifold \((X, g_X)\) and a null-cobordism \(Y\) for \(X\) one has
\begin{equation}
\label{10000}
\begin{aligned}
\varkappa_R(X, g_X) &= \langle (g_X, \tau_Y)^*(r), [Y] \rangle \overset{\text{Lm.~\ref{55}\, (2)}}{=} \langle (\mathfrak{a}, \mathfrak{b})^*(\Phi(r)), [D^{q+1}] \rangle \\
&= \langle (\mathfrak{a}, \mathfrak{b})^*(\rho), M \rangle = \langle \mathfrak{b}^*(\rho^{(0)}), M_{(0)} \rangle + \langle \mathfrak{a}^*(\rho^{(1)}), M_{(1)} \rangle \\
&\overset{\eqref{999}}{=} \langle \sum_{j = 0}^q \Sq^{q-j+1}(\mathfrak{b}^*(\alpha_j)), M_{(0)} \rangle + \langle \delta \mathfrak{b}^*(\omega), M_{(0)} \rangle + \langle \mathfrak{a}^*(\rho^{(1)}), M_{(1)} \rangle \\
&\overset{\eqref{1001},\,\eqref{1002}}{=} \langle \sum_{j = 0}^q \Sq^{q-j+1}(\mathfrak{b}^*(\alpha_j)), M_{(0)} \rangle + \langle \mathfrak{a}^*(\Th (i \circ f)^*(\omega))\vert_\Sigma + \mathfrak{a}^*(\rho^{(1)})\vert_\Sigma, M_{(1)} \rangle.
\end{aligned}
\end{equation}

%% file: quadr.tex
\section{Quadratic property}
\label{sec:quadr}
 Recall that we fix non-negative integers \(n\) and \(q\) such that \(q \ge 2n + 2\), and consider \(n\)-submanifolds of \(\mathbb{R}^q\).
Consider the set of pairs $((X, g_X),x)$ where $(X,g_X)$ is a $(B,f)$-manifold of dimension $n$ (not necessarily compact)  without boundary and $x\in X$. The pairs $((X_1,g_{X_1}),x_1)$ and $((X_2,g_{X_2}),x_2)$ are called equivalent if $x_1 = x_2$ and there exists an open neighborhood $x_1 \in U \subset \R^q$ such that $X_1 \cap U = X_2\cap U$ and $g_{X_1}\vert_{X_1\cap U}=g_{X_2}\vert_{X_2\cap U}$. An equivalence class \([(X, g_X), x]\) is called a \textit{germ} of a $(B,f)$-manifold, and the set of equivalence classes is denoted by $E_B$. Following our convention, if the pair \((B, f)\) is constructed from some relation \(R \in \mathcal{R}_{n+1}\), then we will write \(E_R\) instead of \(E_B\).

Note that there is a well-defined operation \(\mathrm{pr}_2 \colon E_B \to \mathbb{R}^q\) that returns the basepoint of a germ. If \(e \in E_B\) and \(\mathrm{pr}_2(e) = x\) then we write \(e = e_x\). 

For a $(B,f)$-manifold $(X,g_X)$ define its \textit{characteristic function}
\[I_{(X, g_X)} \colon E_B \to \Z/2, \; e_x \mapsto \begin{cases}1, & x \in X\,\text{and}\,e_x = [((X, g_X), x)], \\ 0, & \text{otherwise}.\end{cases}\]

\begin{definition}
For a set-function $f\colon G\to H$ between abelian groups, define its \textit{cross-effect} $d_f\colon G\times G\to H$ by the formula
\[d_f(a,b) \vcentcolon= f(a+b) - f(a) - f(b) +f(0).\]
If $d_f$ is bilinear, then the map $f$ is called \textit{quadratic}.
\end{definition}
Let us now formulate the main result of this section.
\begin{theorem}\label{main} For a relation $R\in\mathcal{R}_{n+1}$ and the corresponding cobordism group $\Omega_n^R$ there is a quadratic map $$\mathfrak{Q}:(\Z/2)^{E_R} \to \Omega_n^R$$ such that for any closed $R$-manifold $(X,g_X)$ the equality $\mathfrak{Q}(I_{(X,g_X)})=[X,g_X]^R$ holds.
\end{theorem}
We will follow the idea of the proof of \cite[Theorem~0.2]{Pod1}.

Let \(\mathcal{X}\) be a collection of closed \((B, f)\)-manifolds. We say that this collection satisfies \textit{the quadratic property (Q)}, if there is a quadratic map \(\mathfrak{Q}_\mathcal{X} \colon (\Z/2)^{E_R} \to \Omega^R_n\) such that \[\mathfrak{Q}_\mathcal{X}(I_{(X, g)}) = [X, g]^R, \quad (X, g) \in \mathcal{X}.\]

\begin{lemma}\label{lem:finite_manifolds}
If every finite collection of closed \(R\)-manifolds satisfies (Q), then the collection of all closed \(R\)-manifolds satisfies (Q).
\end{lemma}
\begin{proof}
Consider the set of all quadratic maps \(Q = \{E_R \to \Z/2\}\) and topologize it with the product topology \(\prod_{e\in E_R}(\Z/2)\). Then by the Tychonoff theorem \(Q\) is a compact space. For every closed \(R\)-manifold \((X, g)\), consider its subspace
\[Q_{(X, g)} = \{A \in Q \mid A(I_{(X, g)}) = [X, g]^R\}.\]
Since every finite intersection \(\bigcap_{i \in I} Q_{(X_i, g_{X_i})}\) is non-empty, the full intersection \(\bigcap_{(X, g)} Q_{(X, g)}\) is non-empty as well.
\end{proof}

From now on, we fix a finite collection \(\{(X_i, g_{X_i})\}_{i \in I}\) of closed \(R\)-manifolds. We shall prove that this collection satisfies (Q). Since our choice of \((X_i, g_{X_i})\)'s is completely arbitrary, this would imply Theorem~\ref{main} by Lemma~\ref{lem:finite_manifolds}.

Consider the linear subspace \(U\) of \((\Z/2)^{E_R}\) generated by the characteristic functions \(u_i \vcentcolon= I_{(X_i, g_{X_i})}\), \(i \in I\). Since there is a linear projection \((\Z/2)^{E_R} \to U\), it is enough to construct a quadratic function \(\mathfrak{Q} \colon U \to \Omega^R_n\) such that \(\mathfrak{Q}(u_i) = [X_i, g_{X_i}]^R\) for every \(i \in I\) to conclude the result.

\subsection{Outline of the construction}

Let \(\{(W_t, g_{W_t})\}_{t \in T}\) be a collection of closed \(R\)-manifolds such that \(\{[W_t]^O\}_{t \in T}\) is a basis of \(\Omega^O_n\). We can assume that the \(W_t\)'s are disjoint from each \(X_i\) and with each other. Write \([X_i]^O = \sum_{t \in T} c_{it} [W_t]^O\), where \(c_{it} \in \Z/2\). Then for every \(i \in I\) we define a bounding \(R\)-manifold \((X^0_i, g_{X^0_i})\) given as a disjoint union:
\[(X^0_i, g_{X^0_i}) = (X_i, g_{X_i}) \cup \bigcup_{c_{it} \ne 0} (W_t, g_{W_t}).\]
Let \(U^0 \subset (\Z/2)^{E_R}\) be the subspace generated by the characteristic functions \(u^0_i \vcentcolon= I_{(X^0_i, g_{X^0_i})}\), \(i \in I\).

Consider an extension \(\Z/2 \hookrightarrow \Omega_n^R \twoheadrightarrow \Omega_n^O\) given by Corollary~\ref{cor:cobordism_group_descr}.

\begin{enumerate}
\item In \ref{sec:construction_varkappa}, we construct a quadratic function \(\widetilde \varkappa_R \colon U^0 \to \Z/2\) such that
\[\widetilde \varkappa_R(u^0_i) = \varkappa_R(X^0_i, g_{X^0_i}), \quad i \in I.\]
\item In \ref{sec:construction_t}, we construct a linear function \(\cob \colon U \to \Omega_n^O\) such that
\[\cob(u_i) = [X_i]^O, \quad i \in I.\]
\item In \ref{sec613}, we construct a linear map \(p \colon U \to U^0\) such that
\[p(u_i) = u^0_i, \quad i \in I.\]
\item In \ref{sec614}, we construct a quadratic section \(s \colon \Omega_n^O \to \Omega_n^R\) to the natural projection such that
\[s([X_i]^O) = \sum_{c_{it} \ne 0} [W_t, g_{W_t}]^R, \quad i \in I.\]
\end{enumerate}
Then we put \(\mathfrak{Q} = \widetilde \varkappa_R \circ p - s \circ \cob\). Note that \(\mathfrak{Q}\) is quadratic, and
\begin{align*}
\mathfrak{Q}(u_i) &= \widetilde \varkappa_R(p(u_i)) - s(\cob(u_i)) = \widetilde \varkappa_R(u^0_i) - s([X_i]^O) \\
&= [X^0_i, g_{X^0_i}]^R - \sum_{c_{it} \ne 0} [W_t, g_{W_t}]^R = [X_i, g_{X_i}]^R, \quad i \in I.
\end{align*}


\subsubsection{Construction of \(\widetilde\varkappa_R\)}
\label{sec:construction_varkappa}

\begin{lemma}[{\cite[Proposition~5.4]{Pod1}}]
\label{lm:powerful}
For sufficiently small compatible tubular neighborhoods of \(X_i^0\)'s there is an open cover \(\Sigma\) of the sphere \(S^q\) and a bilinear map \(\mathfrak{A} \colon U^0 \times C^*(\Th f^* \bar \gamma_n^q) \to C^*(\Sigma)\) such that \(\mathfrak{A}(u_i^0, -) = \mathfrak{a}_i^*(-)\vert_\Sigma\) for every \(i \in I\), where \(\mathfrak{a}_i\) stands for the \((B, f)\)-collapse map for \(X_i^0\) constructed according to the chosen tubular neighborhoods.
\end{lemma}
Let \(\Sigma\) be such an open cover.
For every \(i \in I\) choose some null-cobordism \(Y_i \subset \R^{q+1}_i\) such that \(\partial Y_i = X_i^0\). Let \(\mathfrak{a}_i \colon S^q \to \Th f^* \bar \gamma_n^q\) and \(\mathfrak{b}_i \colon D^{q+1} \to \Th \bar\gamma^{q+1}_{n+1}\) be the collapse maps for \((X_i^0, g_{X_i^0})\) and \(Y_i\), respectively. Let \(\mathfrak{A}\) be a bilinear map from the lemma above. Now, we have to make some choices:
\begin{itemize}
\item choose a relative cocycle \((\rho^{(0)}, \rho^{(1)}) \in Z^{q+1}(\Th f^* \bar \gamma^q_n \to \Th \bar \gamma^{q+1}_{n+1})\) representing \(\Phi(r)\) and let \(\alpha_j \in Z^j(\Th \bar \gamma^{q+1}_{n+1})\) and \(\omega \in C^q(\Th \bar \gamma^{q+1}_{n+1})\) be as in \S\ref{alpha_omega_construction};
\item choose a relative cycle \((M_{(0)}, M_{(1)})\in C_{q+1}(S^q\to D^{q+1})\) representing the fundamental class \([D^{q+1}] \in H_{q+1}(S^q \to D^{q+1})\) as in \S\ref{sec:cycle_choice}.
\end{itemize}

For every \(j = 0, \ldots, q\), consider two linear maps:
\begin{itemize}
\item \(h_j \colon U^0 \to C^*(\Sigma)\) by setting \(h_j(u^0) \defeq \mathfrak{A}(u^0, \Th(i \circ f)^*(\alpha_j))\);
\item \((-)\vert_\Sigma\colon Z^j(D^{q+1}) \to C^j(\Sigma)\).
\end{itemize}

Then
\[h_j(u^0_i) = \mathfrak{A}(u^0_i, \Th(i \circ f)^*(\alpha_j)) \overset{\text{Lm.~\ref{lm:powerful}}}{=} \mathfrak{a}^*_i(\Th(i \circ f)^*(\alpha_j))\vert_\Sigma \overset{\eqref{a_b_commutativity}}{=} \mathfrak{b}^*_i(\alpha_j)\vert_\Sigma.\]
Thus, \(\mathrm{Im}(h_j) \subset \mathrm{Im}((-)\vert_\Sigma)\) for every \(j\). Then there exist linear maps \(J_j \colon U^0 \to Z^j(D^{q+1})\) such that \(J_j(-)\vert_\Sigma = h_j\) for every \(0 \le j \le q\).

Set
\[\widetilde\varkappa_R(u^0) \defeq \sum_{j = 0}^q \underbrace{\langle \Sq^{q-j+1}(J_j(u^0)), M_{(0)} \rangle}_{\text{quadratic part}} + \underbrace{\langle \mathfrak{A}(u^0, \Th(i \circ f)^*(\omega) + \rho^{(1)}), M_{(1)} \rangle}_{\text{linear part}}.\]
The first term is quadratic, since Steenrod squares are defined on the level of cochains in terms of the bilinear operation \(\smile_c\colon\ C^a(D^{q+1})\otimes C^b(D^{q+1})\to C^{a+b-c}(D^{q+1})\): \[\Sq^d(x) \defeq x\smile_{a-d}x, \quad x \in C^a(D^{q+1}).\] 
Note that
\begin{equation}\label{eq_99}
\mathfrak{A}(u^0_i, \Th(i \circ f)^*(\omega) + \rho^{(1)}) \overset{\text{Lm.~\ref{lm:powerful}}}{=} \mathfrak{a}_i^*(\Th(i \circ f)^*(\omega))\vert_\Sigma + \mathfrak{a}_i^*(\rho^{(1)})\vert_\Sigma.
\end{equation}
Also,
\[J_j(u^0_i)\vert_\Sigma = h_j(u^0_i) = \mathfrak{b}_i^*(\alpha_j)\vert_\Sigma,\]
and hence
\begin{equation}
\label{eq_88}
\langle \Sq^{q-j+1}(J_j(u^0_i)), M_{(0)} \rangle = \langle \Sq^{q-j+1}(\mathfrak{b}_i^*(\alpha_j)), M_{(0)} \rangle
\end{equation}
for every \(i\) by Lemma~\ref{lemma_10000000}.
Hence, for every \(i\) one has
\begin{align*}
\widetilde\varkappa_R(u^0_i) &= \sum_{j = 0}^q \langle \Sq^{q-j+1}(J_j(u^0_i)), M_{(0)} \rangle + \langle \mathfrak{A}(u^0_i, \Th(i \circ f)^*(\omega) + \rho^{(1)}), M_{(1)} \rangle \\
&\overset{\eqref{eq_99},\,\eqref{eq_88}}{=} \sum_{j = 0}^q \langle \Sq^{q-j+1}(\mathfrak{b}_i^*(\alpha_j)), M_{(0)} \rangle + \langle \mathfrak{a}^*(\Th (i \circ f)^*(\omega))\vert_\Sigma + \mathfrak{a}^*(\rho^{(1)})\vert_\Sigma, M_{(1)} \rangle \\
&\overset{\eqref{10000}}{=} \varkappa_R(X_i^0,g_{X_i^0}).
\end{align*}

\subsubsection{Construction of \(\cob\)}
\label{sec:construction_t}
By the theorem of Thom, the unoriented cobordism class of a manifold is determined by its Stiefel--Whitney numbers. Hence, it suffices to construct for every \(z \in H^n(\Gr_n(\R^q))\) an additive function \(\cob_z \colon U \to \Z/2\) such that \(\cob_z(u_i) = \langle \tau_{X_i}^*(z), [X_i] \rangle\) for every \(i \in I\). Let \(\theta \in \widetilde Z^q(\Th f^* \bar \gamma^q_n)\) be a cocycle representing \(\varphi(f^*(z)) \in \widetilde H^q(\Th f^* \bar \gamma^q_n)\). Then we can set
\[\cob_z(u) \defeq \langle \mathfrak{A}(u, \theta), M_{(1)} \rangle.\]
Then indeed,
\begin{align*}
\cob_z(u_i) &= \langle \mathfrak{A}(u_i, \theta), M_{(1)} \rangle \overset{\text{Lm.~\ref{lm:powerful}}}{=} \langle \mathfrak{a}^*_i(\theta)\vert_\Sigma, M_{(1)} \rangle \\
&= \langle \mathfrak{a}^*_i(\theta), M_{(1)} \rangle = \langle \mathfrak{a}^*_i(\varphi(f^*(z))), [S^q] \rangle \\
&\overset{\text{Lm.~\ref{55}\,(1)}}{=} \langle g^*_{X_i}(f^*(z)), [X_i] \rangle = \langle \tau^*_{X_i}(z), [X_i] \rangle.
\end{align*}

\subsubsection{Construction of \(p\)} \label{sec613}
Consider linear functionals \(\cob_t \colon U \to \Z/2,\) \(t \in T\), such that
\[\cob(u) = \sum_{t \in T} \cob_t(u) [W_t]^O, \quad u \in U.\]
Note that for every \(i \in I\) one has
\[\sum_{t \in T} c_{it} [W_t]^O = [X_i]^O = \cob(u_i) = \sum_{t \in T} \cob_t(u_i) [W_t]^O.\]
Hence, \(c_{it} = \cob_t(u_i)\) for every \(i \in I\) and \(t \in T\).

Then, we can define a linear map \(p \colon U \to U^0\):
\[p(u) \vcentcolon= u + \sum_{t \in T} \cob_t(u) I_{(W_t, g_{W_t})}.\]
One can verify that
\[p(u_i) = u_i + \sum_{t \in T} \cob_t(u_i) I_{(W_t, g_{W_t})} = u_i + \sum_{t \in T} c_{it} I_{(W_t, g_{W_t})} = I_{(X^0_i, g_{X^0_i})} = u^0_i.\]

\subsubsection{Construction of \(s\)}\label{sec614}

Let \(\mathrm{sq} \colon \Z/2 \to \Z/4\) be defined by \(0 \mapsto 0, 1 \mapsto 1\). This map is quadratic since there is the following commuting diagram:
\[\begin{tikzcd}[cramped]
	\Z & \Z \\
	{\Z/2} & {\Z/4.}
	\arrow["{(-)^2}", from=1-1, to=1-2]
	\arrow["{\mathrm{mod}\;2}"', from=1-1, to=2-1]
	\arrow["{\mathrm{mod}\;4}", from=1-2, to=2-2]
	\arrow["{\mathrm{sq}}"', from=2-1, to=2-2]
\end{tikzcd}\]
Note that \(\Omega^R_n\) is a module over \(\Z/4\) by \eqref{eq:cobordism_ses}. Then there is a quadratic section \(s \colon \Omega^O_n \to \Omega^R_n\) to the natural projection given by
\[s(\sum_t a_t [W_t]^O) = \sum_t \mathrm{sq}(a_t) [W_t, g_{W_t}]^R.\]
Then
\[s([X_i]^O) = s(\sum_{c_{it} \ne 0} [W_t]^O) = \sum_{c_{it} \ne 0} [W_t, g_{W_t}]^R.\]